\documentclass[leqno,12pt]{article}
\usepackage{amssymb,amsmath}
\usepackage{latexsym}
\usepackage [dvips]{graphics,psfrag}
\usepackage{amsfonts}
\usepackage{url}
\usepackage{rotating}
\usepackage{pdfsync}
\usepackage{mathrsfs}
\usepackage{graphicx,graphics,epsf,psfrag,eepic,epic}
\usepackage{xy}
\xyoption{all}
\input epsf.tex
\def\bases{\operatorname{Bases}}
\newtheorem{theorem}{Theorem}
\newtheorem{proposition}[theorem]{Proposition}

\newtheorem{definition}[theorem]{Definition}

\newtheorem{remark}[theorem]{Remark}
\newtheorem{example}[theorem]{Example}

\newenvironment{proof}{{\bf Proof.}}{\par}

\newcommand{\C}{{\mathbb C}}
\newcommand{\R}{{\mathbb R}}
\newcommand{\Z}{{\mathbb Z}}
\newcommand{\N}{{\mathbb N}}

\renewcommand{\ll}{{\langle}}
\newcommand{\rr}{{\rangle}}

\newcommand{\gc}{{\mathfrak{c}}}

\newcommand{\vol}{\operatorname{vol}}
\renewcommand{\a}{{\mathfrak{a}}}

\renewcommand{\c}{{\mathfrak{c}}}

\def\proj{\operatorname{proj}}

\def\h{\operatorname{ht}}

\newcommand{\g}{{\mathfrak{g}}}

\renewcommand{\t}{{\mathfrak{t}}}
\renewcommand{\k}{{\mathfrak{k}}}

\newcommand{\CA}{{\cal A}}
\newcommand{\CC}{{\cal C}}

\newcommand{\CR}{{\cal R}}

\newcommand{\CH}{{\cal H}}

\newcommand{\CP}{{\cal P}}

\newcommand{\CV}{{\cal V}}
\newcommand{\CW}{{\cal W}}

\newcommand{\la}{{\langle}}
\newcommand{\ra}{{\rangle}}

\newcommand{\res}{{\mathrm{res}}}
\newcommand{\Ires}{{\mathrm{Ires}}}\def\JK{\operatorname{JK}}

\title{Discrete series representations and $K$ multiplicities for  $U(p,q).$ User's guide}
\author{Velleda Baldoni and Mich{\`e}le Vergne}
\date{}

\begin{document}
\maketitle

\begin{abstract}This document is a companion for the Maple program
\textbf{Discrete series and $K$-types for  $U(p,q)$} available on \begin{verbatim} http://www.math.jussieu.fr/~vergne/ \end{verbatim}

We explain an algorithm to compute the multiplicities of an irreducible representation of $U(p)\times U(q)$ in a discrete series of $U(p,q)$.  It is based on Blattner's formula. We    recall the general mathematical
background to compute Kostant partition functions via multidimensional residues, and  we outline our  algorithm.
We also point out some properties of the piecewise polynomial functions describing multiplicities based on Paradan's results.
\end{abstract}

{\small


\tableofcontents

\section*{Introduction}
The present article is a user's guide for the Maple program
\textbf{Discrete series and K types  for  $U(p,q)$}, available at
\begin{verbatim} http://www.math.jussieu.fr/~vergne/ \end{verbatim}

 In the first part, we explain what our program does  with simple examples.
The second  part sets the general mathematical background.
We recall Blattner's formula, and we discuss the piecewise   (quasi)-polynomial behavior of multiplicities for a discrete series of a reductive real group.
In the third part, we outline the algorithm of computing partition functions for arbitrary set of vectors, based on Jeffrey-Kirwan residue and maximally nested subsets.
In the fourth part, we specialize the method  to $U(p,q)$.
We in particular explain how to  compute maximally nested subsets
of the set of non compact positive roots of $U(p,q)$.

In the last  two sections, we give more examples and some details on the implementation of the algorithms for our present application.

\bigskip

Here  $G$ is the Lie group $G=U(p,q)$  and $K=U(p)\times U(q)$ is a maximal compact subgroup of $G$.
Of course all these issues  can be addressed for the other reductive real Lie groups following the same approach described here.
To introduce the function we  want to study, we briefly establish some notations (see Sec. \ref{BF}). We denote by  $\pi^\lambda$  a discrete series representation with Harish Chandra parameter $\lambda$. The  restriction  of $\pi^{\lambda}$ to the maximal compact subgroup $K$ decomposes in irreducible finite dimensional  $K$ representations  with finite multiplicities, in formula
  $$\pi^{\lambda}\ _{|_{K}}= \sum_{\tau_\mu}m^{\lambda}_{\mu} \tau_{\mu}$$
  where we sum over  $\hat K$: the classes  $\tau_\mu$  of  irreducible finite dimensional $K$ representations, $\mu$ being the Harish Chandra parameter of $\tau_\mu.$

  \noindent $m^{\lambda}_{\mu} $ is a finite number  called  the {\bf multiplicity}  of the $K$-type $\tau_{\mu},$ or simply of  $\mu,$  in $\pi^\lambda$  and this paper addresses the question of computing $m^{\lambda}_{\mu} $.

\noindent The algorithm described in this paper
checks whether a certain $K$-type $\tau_\mu$  appears in the $K$-spectrum of a discrete series $\pi^{\lambda}$  by computing the multiplicity of such a $K$-type.
The input $\mu$ can also  be a symbolic  variable  as we will explain shortly.

\noindent By Blattner's formula,   computing the K-multiplicity is equivalent to compute the number of integral points, that is a partition function,  for  specific polytopes. We thus  use the formulae developed in \cite{BBCV} to write the algorithm.

One important aspect of these results is that the input datas $(\lambda,\mu)$ can also be treated as parameters.
Thus, in principle,  given $(\lambda_0,\mu_0)$, we can  output
a convex cone, containing $(\lambda_0,\mu_0)$, and a polynomial function  of the parameters $(\lambda,\mu)$  whose value at $(\lambda,\mu)$ is the multiplicity of the $K$-type $\mu$ in the representation $\pi^{\lambda}$ as long as we stay within the region described by the cone and $(\lambda,\mu)$ satisfy some integrality conditions to be defined later.
We also plan to explicitly decompose our parameter space $(\lambda,\mu)$
in  such regions of polynomiality in a future study, thus describing fully the piecewise polynomial function $m_\mu^{\lambda}$, at least for some low rank cases.

In practise here, we  only address a simpler question: we  will fix $\lambda$, $\mu$  and a direction $\vec{v}$  and compute a piecewise polynomial function of $t$ coinciding with $m_{\mu+t\vec{v}}^{\lambda}$ on integers $t$.
This way, we can   also check if  a direction $\vec{v}$ is an asymptotic direction of the $K$-spectrum, in the sense explained in  Sec. \ref{asymp}.

The Atlas of Lie Groups and Representations, \cite {ATLAS},   within the problem of classifying all of the irreducible unitary representations of a given reductive Lie group,  addresses in particular the problem of computing K-types of discrete series.  The multiplicities results needed for the general unitary problem  is of different nature as we are going to  explain.

Given as input  $\lambda$ and some height $h,$
 Atlas computes
{\bf the list with multiplicities}  of all the  representations occurring in $\pi^\lambda$ {\bf of height  smaller} than $h$.
But the efficiency, in this setting, is  limited by the height.
In contrast, the efficiency  of our program is  unsensitive to the height of $\lambda,\mu$, but the output
is  {\bf  one number}: the multiplicity of $\mu$ in $\pi^\lambda$.
It takes (almost) the same time to compute the multiplicity of the lowest
$K$-type of $\pi^\lambda$ (fortunately the answer is $1$)
than the multiplicity of a representation of very large height.
Our calculation are also very sensitive to the rank: $p+q-1$.

For other applications  (weight multiplicities, tensor products multiplicities) based  on computations of  Kostant  partition functions in the context of finite dimensional representations,  see \cite{C1}, \cite{C2},  \cite{BL}.

\section{The algorithm for Blattner's formula:  main  commands and simple  examples }

\noindent Let $p,q$  be integers. We consider the group $G=U(p,q)$.
The maximal compact subgroup is $K:=U(p)\times U(q)$.
More details in parametrization are given in Section \ref{Upq}.

A discrete series representation $\pi^\lambda$ is parametrized according to Harish-Chandra parameter $\lambda$, that we input as {\bf discrete}:
$$discrete:= [[\lambda_1,\ldots, \lambda_p],[\gamma_1,\ldots,\gamma_q]].$$

 Here $\lambda_i,\gamma_j$ are {\bf integers} if $p+q$ is odd, or {\bf half-integers} if $p+q$ is even. They are all distinct.
 Furthermore
 $ \lambda_1> \cdots > \lambda_p$ and  $\gamma_1> \cdots >\gamma_q.$

A  unitary irreducible representation of $K$ (that is a couple of
unitary irreducible representations of $U(p)$ and  of $U(q)$)  is parametrized by its Harish-Chandra parameters $\mu$ that we input as {\bf Krep}:
$$Krep:=[[a_1, a_2,\ldots,a_p],[b_1,\ldots,b_q]]$$
with $a_1>\cdots > a_p$
and $b_1>\cdots> b_q$.

Here  $a_i$ are integers if $p$  is odd, half-integers if $p$ is even.
Similarly $b_j$ are integers if $q$ is odd, half-integers if $q$ is even.

As we said, our   objective  is to study the function  $m_{\mu}^{\lambda}$  for  $
\mu \in \hat K$
 where
   $\mu=Krep$ and $\lambda=discrete.$

\noindent The examples are runned on a MacBook Pro, Intel Core 2 Duo, with a Processor Speed of 2.4 GHz. The time of the examples is computed in seconds. They are recorded  by
\begin{verbatim} TT \end{verbatim}.

Some of these examples are very simple and can be checked by hand (as we did, to reassure ourselves). Other examples are given at the end of this article.

\bigskip

 To compute the multiplicity of the $K$-type given by {\bf Krep}
 in the discrete series with parameter  given by {\bf discrete}, the  {\bf command} is
\begin{verbatim}
>discretemult(Krep,discrete,p,q)
\end{verbatim}
\begin{example}\label{Ex1}\end{example}
{\scriptsize
\begin{verbatim}
Krep:=[[207/2, -3/2], [3/2, -207/2]];
discrete:=[[5/2, -3/2], [3/2, -5/2]];

>discretemult(Krep,discrete,2,2);

101
\end{verbatim}}
Here is another example of    $m_\mu^\lambda$  that our program can compute.

 \begin{example}\end{example}
{\scriptsize
\noindent We consider the discrete series indexed by
\begin{verbatim}
 lambda33:=[[11/2, 7/2, 3/2], [9/2, 5/2, 1/2]];
\end{verbatim}
Its lowest $K$-type is
\begin{verbatim}
 lowlambda33:= [[7,4,1], [5,2,-1]];
\end{verbatim}

\noindent Of course, the multiplicity of the lowest $K$-type is $1$. Our programm fortunately returns the value $1$ in $0.03$ seconds.

\medskip \noindent Consider now the representation of $K$ with parameter:
\begin{verbatim}
 biglambda33:= [[10006, 4, -9998], [10004, 2, -10000]];
\end{verbatim}
Then the multiplicity of this $K$-type is computed  computed  in $0.05$ seconds as
\begin{verbatim}
> discretemult(biglambda33,lambda33,3,3);

  2500999925005000;
\end{verbatim}}
Here are two other examples that verify the known behavior of holomorphic discrete series.
The notation $ab...$ that we use to label  the discrete series parameters is introduced in \ref{noncp} and it  very effective to picture the situation,
 but it is not relevant to understand the following computation.
 \begin{example}\end{example}
{\scriptsize Consider a holomorphic discrete series  of type "aaabbb" for $G=U(3,3)$ (see Sec.\ref{noncp}) with lowest $K$-type of dimension $1$. We verify that the multiplicity is $1$ for $\mu$ in the cone spanned by strongly orthogonal non compact positive roots.

\begin{verbatim}
hol33:=[[11/2,9/2,7/2],[5/2,3/2,1/2]];
lowhol33:=[[7, 6, 5], [1, 0, -1]];
bighol33:=[[7+1000, 6+100, 5+10], [1-10, -100, -1-1000]];

>discretemult(lowhol33,hol33,3,3);

1
TT:=  0.052

>discretemult(bighol33,hol33,3,3);

1
TT:=  1.003
\end{verbatim}}
\begin{example}\end{example}
{\scriptsize
Consider a holomorphic discrete series  of type "aaabbb" for $G=U(3,3)$ (see Sec. \ref{noncp}) with lowest $K$-type of dimension $d$. We verify that the multiplicities are bounded by $d$.
\begin{verbatim}
Hol33:=[[27/2,9/2,7/2],[5/2,3/2,-5/2]];
lowHol33:=[[15, 6, 5], [1, 0, -4]];
bigHol33:=[[15+1000, 6+1000, 5+1000], [1-1000, -1000, -4-1000]];
verybigHol33:=[[15+100000, 6+10000, 5+10000], [1-10000, -10000, -4-100000]];

>discretemult(lowHol33,Hol33,3,3);

 1
TT:= 0.069

>discretemult(bigHol33,Hol33,3,3);

 4
TT:=0.971

>discretemult(verybigHol33,Hol33,3,3);

4
TT:= 0.873
\end{verbatim}
}

\bigskip

Fix now a K-type $\mu,$ a direction $\vec v$ given by a dominant weight for $K$
( more details in Sec. \ref{asymp}) and  a discrete series parameter $\lambda$.
  The half-line $\mu+t\vec v$ stays inside the dominant chamber for $K$.
  A very natural  question is that of  investigating the behavior of the  multiplicity function as a function of $t\in \Z$, when we move from $\mu$ along  the positive $\vec v$ direction, that is the function
 $t \rightarrow m_{\mu+t\vec v}^\lambda, \ t\geq 0.$ The answer to this question will be given by two sets of datas: a  covering  of $\N$ determined by   a finite number of closed intervals $I_i\subset \R$ with integral end points, that is  $\N= \cup _{1\leq i\leq s}(I_i\cap \N)$,  together with  polynomial functions $P_i(t), \ 1\leq i\leq s,$   of degree bounded by $pq-(p+q-1)$,
  that compute  the multiplicity  on such  intervals $I_i$: in formula $m_{\mu+t\vec v}^\lambda=P_i(t)$ for $t\in I_i\cap \N.$

 \noindent We remark two aspects.
 First the $\{I_i\cap \N, \ 1\leq i\leq s\}$ constitutes a covering in the sense that we recover all of $\N$ but  $(I_j\cap \N)\cap(I_{j+1}\cap \N)$ can intersect in the extreme points and hence in this case $P_j$ and $P_{j+1}$ have to coincide on the intersection.  Secondly the  "intervals" $I_i\cap \N$   can be reduced   just to a point, so that the polynomial $P_i$ (if not constant) is not uniquely determined by its value on one point !. More generally, if the length of the interval $I_i$ is smaller that the degree of $P_i$, the polynomial $P_i$ is not uniquely determined.

\noindent Because of our  focus on the polynomiality aspects we keep  in the output of the algorithm the polynomials $P_i$ even if the intervals $I_i$ are   reduced to a point (or with small numbers of integral points).

\noindent In our application, $\mu$ will be the lowest K-type and  we will give some examples of the  situations occurring, in particular we will examine the following cases:
 \begin{itemize}
 \item The first example outputs  a  covering of $\N$  given by a unique  interval and a polynomial function on $\N$ that computes the multiplicity. Thus in this case,  all the integer points on the half  line $\mu+t\vec v, t\geq 0$ give rise to K-types that appear in the restriction of the discrete series, in particular $\vec v$ is an asymptotic direction,  (see Sec.  \ref{defas}  and \ref{asymp}),
  \item In the other examples, the  covering of $\N$ has  at least two intervals and illustrate different situations.
 \end{itemize}
 To compute, in the sense we just explained,   the multiplicity of the  $K$-type $\mu+t\vec v$
 in the discrete series with parameter  {\bf discrete}, $\mu$ being the lowest $K$-type moving in the  positive direction $\vec v$, labeled by {\bf direction} ,  the  {\bf command} is
\begin{verbatim}
>function_discrete_mul_direction_lowest(discrete,direction,p,q);


\end{verbatim}
\begin{example}\end{example}\label{Example1}
{\scriptsize
\begin{verbatim}
discrete:=[[5/2, -3/2], [3/2, -5/2]];
direction:=[[1,0],[0,-1]];

>function_discrete_mul_direction_lowest(discrete,[[1,0],[0,-1]], 2,2);

 [[t+1, [0, inf]]]
\end{verbatim}
($inf$ stands for  $\infty$ ).  }

\noindent Here the covering of $\N$ is $\N=[0,\infty]\cap \N$ and the polynomial is $P(t)=t+1.$
\medskip
The output explicitly compute:
$$m_{\mu+t\vec v} ^\lambda= t+1, \  t\in \N, \ t\geq 0$$
with   $\mu=[[7/2,-3/2],[3/2,-7/2]]$  the lowest $K$-type of the representation $\pi^{\lambda}$
(See Ex.\ref{Ex2} for computation of  the lowest $K$-type).
 Thus we compute the multiplicity, starting from the lowest $K$-type when we  are moving off it in the direction of $\vec{v}$.
In particular for $t=100$ we get $m_{\mu+100\vec v} ^\lambda=101$ as predicted in Ex.\ref{Ex1}, since
 \begin{verbatim}
discrete:=[[5/2, -3/2], [3/2, -5/2]]
\end{verbatim}
and   $\mu+100 \vec v$ is equal to

 \begin{verbatim}
 Krep:=[[207/2, -3/2], [3/2, -207/2]];
 \end{verbatim}

 \begin{example}\end{example} \label{Example2}
{\scriptsize
\begin{verbatim}
discrete:=[[9, 7], [-1, -2, -13]];
direction:= [[1, 0], [0, 0, -1]];

>function_discrete_mul_direction_lowest(discrete, direction,2,3);

[[1+(1/2)*t-(1/2)*t^2, [0, 0]], [1, [1, inf]]]
\end{verbatim}}

\noindent Here the covering is $\N=([0,0]\cap \N)\cup ([1,\infty]\cap \N$ and the polynomials are $P_1(t)=1+\frac12t-\frac12t^2$ on $[0,0]\cap \N$ and $P_2(t)=1$ on
$([1,\infty]\cap \N.$
\noindent  Observe that $[0,0]\cap \N=[0]$ is just a point and that  $P_1(0)=1$ as it should be, since $\mu$ is the lowest $K$-type.
Explicitly we simply compute

$$m_{\mu+t\vec v} ^\lambda= 1, \  t\in \N, \ t\geq 0.$$

We conclude with one example in which the multiplicity grows:  the last polynomial is not  zero and has degree two.  We give more examples at the end of this article.

 \begin{example}\label{1}\end{example}
{\scriptsize
\begin{verbatim}
discrete:= [[57/2, 39/2, 3/2], [51/2, 5/2, -155/2]];
direction:=[[1, 0, 0], [0, 0, -1]];

>function_discrete_mul_direction_lowest(discrete,direction,3,3);

[(1/24)*t^4+(5/12)*t^3+(35/24)*t^2+(25/12)*t+1, [0, 16]], [-3059+(2242/3)*t-(133/2)*t^2+(19/6)*t^3, [17, 21]],
[-11914+(9597/4)*t-(4367/24)*t^2+(27/4)*t^3-(1/24)*t^4, [22, 40]], [100016-8664*t+228*t^2, [41, inf]]
 \end{verbatim}}
That is  for $\lambda=$discrete and $\mu=[[30, 20, 1], [26, 2, -79]]$ the lowest $K$-type
 \[
m_{\mu+t\vec v} ^\lambda=\left \{   \begin{tabular}{lc}
$(1/24)*t^4+(5/12)*t^3+(35/24)*t^2+(25/12)*t+1$&  $0\leq t \leq 16$ \\
$-3059+(2242/3)*t-(133/2)*t^2+(19/6)*t^3 $&$17\leq t \leq 21$ \\
$-11914+(9597/4)*t-(4367/24)*t^2+(27/4)*t^3-(1/24)*t^4$ &$22\leq t \leq 40$ \\
$100016-8664*t+228*t^2$ &$ t \geq 41$
   \end{tabular}\right.
\]
The time to compute the example is $TT := 0.835$ and the formula says for instance that
$m_{\mu+2000000\vec v} ^\lambda=911982672100016$

\bigskip

To compare with other parametrizations of the discrete series representations,
it may be useful also to give here the command for the lowest $K$-type of  the discrete series  with  parameter $\lambda=discrete$
of the group $U(p,q)$. The command  is:

\begin{verbatim}
>Inf_lowestKtype(p,q, discrete)
\end{verbatim}
 \begin{example}\label{Ex2}\end{example}
 {\scriptsize
\begin{verbatim}

> Inf_lowestKtype([[5/2,-3/2],[3/2,-5/2]],2,2);

 [[7/2,-3/2], [3/2, -7/2]]
\end{verbatim}}

Similarly, we may want to parametrize a representation of $K$ by its highest weight.
Then the {\bf command} is:
\begin{verbatim}
> voganlowestKtype(discrete,p,q)
\end{verbatim}
\begin{example}\end{example}
{\scriptsize
\begin{verbatim}
> voganlowestKtype([[5/2,-3/2],[3/2,-5/2]],2,2);

 [[3,-1], [1, -3]]
\end{verbatim}}

  Let us finally  recall the simple case of  the multiplicity function for  $G=U(2,1)$ (see Sec. \ref {noncp} for the notations).

 Choose as positive compact root the root $e_1-e_2$ and denote by  $w$  the corresponding simple reflection. Fix $\lambda=[[\lambda_1,\lambda_2],[\lambda_3]]$ the Harish Chandra parameter for a discrete series  representation. We assume $\lambda$ regular and $\lambda_1>\lambda_2.$ There are three chambers $\c_1$, $\c_2$, $\c_3$ and hence three systems of positive roots containing $e_1-e_2.$  Precisely $\c_1$ corresponds to the positive system $\{e_1-e_2, e_1-e_3, e_2-e_3\}$, $\c_2$ corresponds to the positive system $\{e_1-e_2, e_1-e_3, e_3-e_2\}$ and $\c_3$ to
   $\{e_1-e_2, e_3-e_1, e_3-e_2\}$.

   We examine the situation in which the discrete series parameter belongs to one of these  chambers.  Fig.\ref{chamber1} and  Fig.\ref {chamber2},    picture  the two chambers $\c_1$, $\c_2$  and evidentiate the values for $m_{\mu}^{\lambda}$ when $\lambda$ is in the chamber.
The black lines mark the  chambers containing the compact root $e_1-e_2$  and the red ones the system of  positive roots for the given chamber.

\begin{figure}[h]
 \begin{center}
  \includegraphics[]{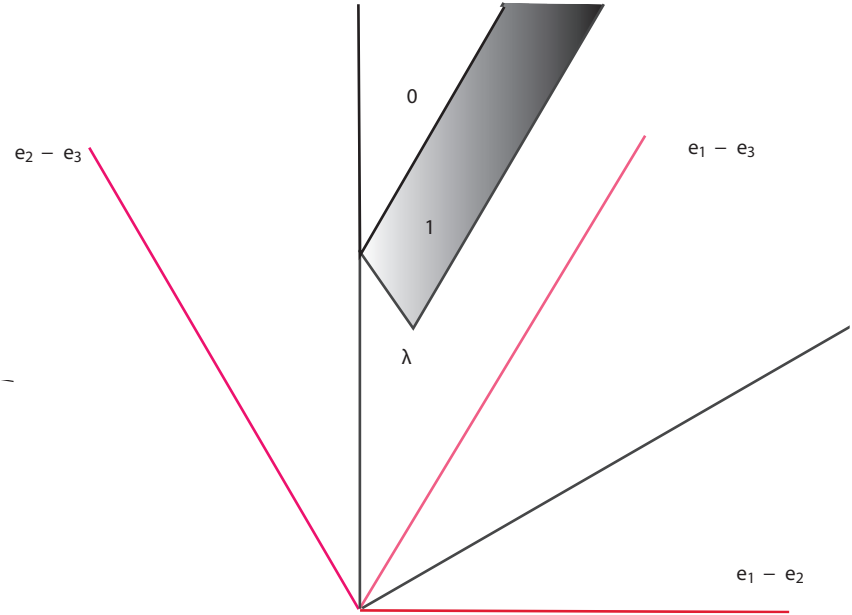}
\caption{ $m_{\mu}^{\lambda}$ for the  chamber $\c_1$ of U(2,1). \label{chamber1}}
\end{center}
\end{figure}
\begin{figure}[h]
 \begin{center}
 \includegraphics[scale=0.7]{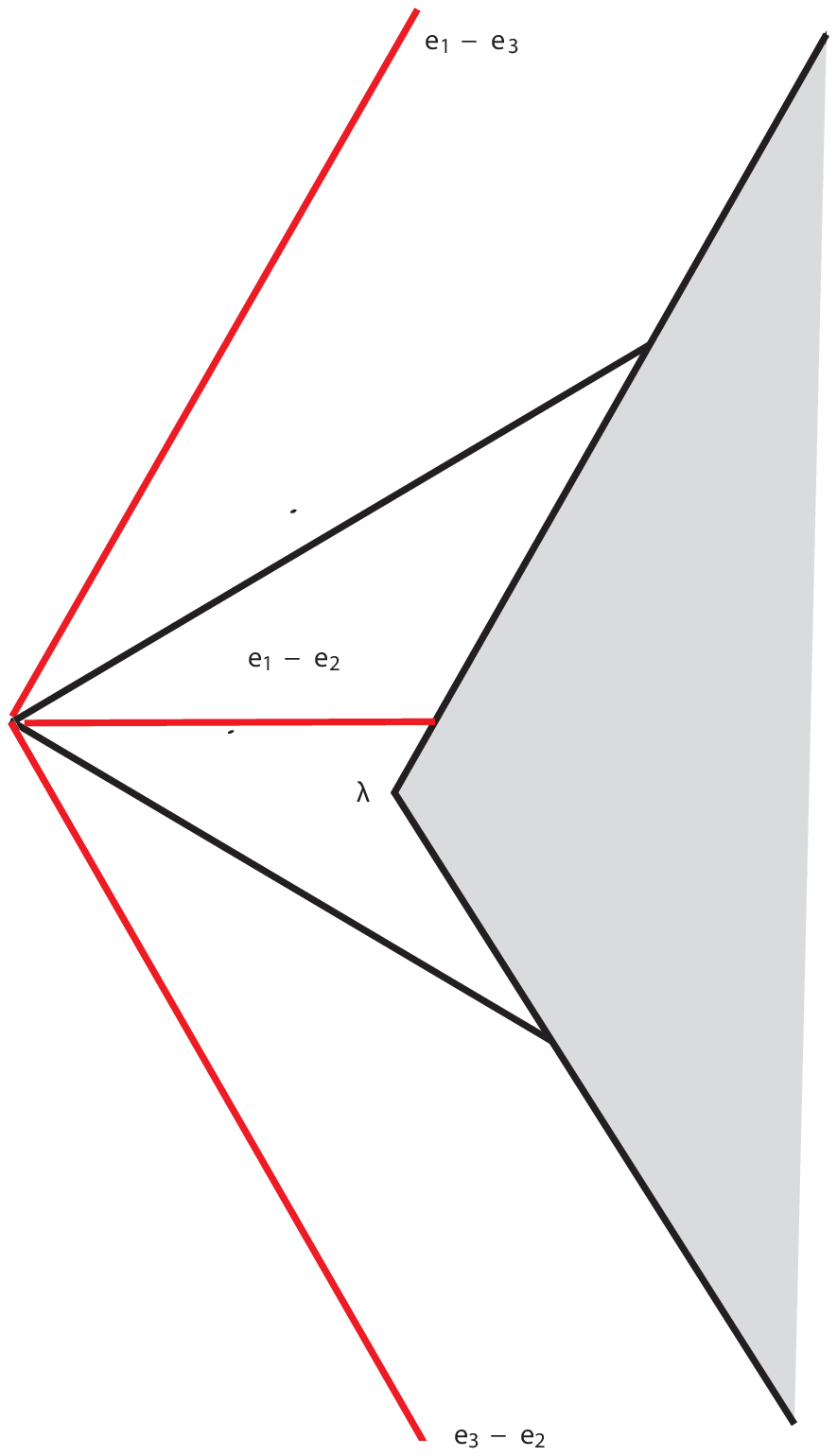}
\caption{ $m_{\mu}^{\lambda}$ for the  chamber $\c_2$ of U(2,1)\label{chamber2}}
\end{center}

\end{figure}

We give some examples concerning the two situations.
The parameter $\lambda$ in
  Fig.\ref{chamber1} is $$hol21 := [[2, 1], [-3]]$$

  In Figure \ref{chamber2}
  the  parameter $\lambda$ is
$$aba:=  [[2, -3], [1]].$$

\newpage \section{ Mathematical background}\label {math}

\subsection{ Notations}
\[\begin{array}{ll}
 G,\g& \mbox{semi-simple real Lie group, Lie algebra of G.}\\
K,\k& \mbox{maximal compact subgroup of $G$, Lie algebra of $K$.}\\
T,\t& \mbox{maximal torus of $K$, Lie algebra of $T$.}\\
P\subset \t^*& \mbox{lattice of weights of $T$.}\\

\Delta^+, \Delta^+(\lambda)\subset \t^*&\mbox{system of positive roots for $G$.}\\
\Delta_c^+, \Delta_c^+(\lambda) &\mbox{system of positive compact roots.}\\
\Delta_n^+, \Delta_n^+(\lambda) &\mbox{system of positive non compact roots.}\\
\Delta^+(A,B) &\mbox{system of positive  roots of parabolic type determined by $(A,B)$.}\\
\a, \a_c &\mbox{positive chamber for  $G$ and $K$ respectively}\\
\rho,\rho_c^+,\rho_n^+&\mbox{half the sum of positive, positive compact,  positive non compact roots.}\\
P_\g,P_\k& \mbox{set of  $G$ admissible, $K$ admissible parameters.}\\
P_\g^r,P_\k^r& \mbox{set of  $G$ admissible and regular , $K$ admissible and  regular  parameters.}\\
U &\mbox{r-dimensional real vector space; $x \in U$.}\\
V &\mbox{dual vector space of $U$,  $h \in V$.}\\
\ll \;,\; \rr &\mbox{the pairing between $U$ and $V$.}\\
V_\Z &\mbox{ lattice  of $V$.}\\
U_\Z &\mbox{ dual lattice in $U$.}\\
\mathcal A^+ &\mbox{a sequence of  vectors in $V_\Z$; $\alpha\in \mathcal A^+$.}\\
\tau &\mbox{a tope.}\\
T& \mbox{torus $U/U_\Z$;  $t\in T$.}\\
F &\mbox{ finite subset of   $T$.}\\
\Pi_{\CA^+}(h)& \mbox{ polytope defined by $\CA^+$.}\\
N_{\CA^+}(h)&\mbox{ number of integral points  for   $\Pi_{\CA^+}(h).$}\\
\c &\mbox{chamber.}\\
\JK_{\gc}&\mbox{ Jeffrey-Kirwan  residue}.\\
\Ires&\mbox{ iterated residue}.\\
\end{array}\]

 \subsection{Blattner formula and multiplicities}\label{BF}

 Let $ G$ be a reductive  connected  linear Lie  group  with  Lie algebra $\mathfrak{g}$ and denote by $K$ a maximal compact subgroup of $G$ with Lie algebra $\mathfrak{k}.$

We assume that  the ranks of $G$ and $K$ are equal.
Under this hypothesis  the group $G$ has discrete series representations.
Recall  Harish-Chandra's parametrization of discrete series representations.
We choose  a compact Cartan subgroup  $T\subset K$ with Lie algebra $\mathfrak t$.
Let $P\subset \t^*$ be the lattice  of weights of $T$.
They correspond to characters of $T$. Here if $\lambda\in P$, the corresponding  character of $T$ is $e^{i\lambda}$.
\noindent Let $\Delta^+\subset P$ be a positive system of roots
and $ \rho=\frac{1}{2}\sum_{\alpha \in \Delta^+} \alpha.$
Then the
subset
$\rho + P \subset \t^*$  does not depend of the choice of the positive system $\Delta^+$.
 We denote it by $P_\g$. We denote by $P_\g^r\subset P_\g$ the subset of
$\g$-regular elements. We can similarly define $P_\k^r.$ We denote by $\mathcal {\mathcal W}_c$ the Weyl group of $K$.
For any $\lambda \in P_\g^r$,
Harish-Chandra defined a discrete series representation
$\pi^\lambda$. Elements of $P_\g^r $ are called  Harish-Chandra parameters for $G$.
Two  representations, $\pi^{\lambda}$ and $\pi^{\lambda'}$,  coincide precisely when their parameters $\lambda,\lambda'$  are related by an an element of $\mathcal {\mathcal W}_c.$
Thus the set of discrete series representations is parametrized by $P_\g^r/\mathcal {\mathcal W}_c.$

In the same way we can parametrize the set $\hat K$ of classes of  irreducible finite dimensional representations  of $K$ by their Harish-Chandra parameter $\mu\in P_\k^r/\mathcal {\mathcal W}_c.$
  Once a positive system of compact roots is chosen, an element $\mu\in P_\k^r$ can be conjugated to a unique  regular element in the corresponding  positive chamber $\a_c\subset \t^*$ for the compact roots. We denote by $\tau_\mu\in \hat K$, or simply by $\mu \in \hat K$, the corresponding representation.

 A discrete series representation $\pi^{\lambda}$ is $K$-finite:
  $$\pi^{\lambda}\ _{|_{K}}= \sum_{\tau_\mu \in \hat K}m^{\lambda}_{\mu} \tau_{\mu}.$$
To determine $m^{\lambda}_{\mu} $,  that is the multiplicity of the $K$-type $\tau_{\mu}$  in $\pi^\lambda$,  is a basic problem in representation theory.

Blattner's formula, \cite {HS}, gives an answer to this problem. We need to introduce a little more notations before stating it.

We let $\Delta$ be the root system for $\mathfrak g$ with respect to $\mathfrak t$, $\Delta_c$ the system of compact roots , that is the roots of  $\mathfrak k$ with respect to $\mathfrak t$, and $\Delta_n$  the system of noncompact roots.
We let $\Delta^+$ be the unique positive system for $\Delta$ with respect to which $\lambda$ is dominant.  
We write $ \rho=\frac{1}{2}\sum_{\alpha \in \Delta^+} \alpha$ , $ \rho_c=\frac{1}{2}\sum_{\alpha \in \Delta^+_c} \alpha$ and $ \rho_n=\frac{1}{2}\sum_{\alpha \in \Delta^+_n} \alpha$ where
$\Delta^+_c=\Delta^+ \cap \Delta_c$ and $\Delta^+_n=\Delta^+\cap \Delta_n$.
Therefore $P_\g=\rho+P$, $P_\k=\rho_c+P$ and if $\xi\in P_\g$, then $\xi+\rho_n\in P_\k$.

Write  $\a,\a_c\subset \t^*$ for the  (closed) positive chambers corresponding to  $\Delta^+$ and  $\Delta_c^+$.
We will also write $\Delta^+(\lambda), \Delta^+_c(\lambda) $  and $\Delta_n^+(\lambda)$ for  $\Delta^+, \Delta^+_c, \Delta^+_n$ if necessary to stress that these systems  depend on $\lambda$.

\noindent  Then for   $\mu \in  P_\k^r\cap \a_c$,  {\bf Blattner's formula} says:
 \begin{equation}  \label{blattner}
m^{\lambda}_{\mu} =\sum_{w \in \mathcal W_{c}} \epsilon(w)  \mathcal P_n(w\mu-\lambda-\rho_n)
\end{equation}
\noindent where, given $\gamma \in \mathfrak t^*$, we define $\mathcal P_n(\gamma)$ to be the number of distinct ways in which $\gamma$ can be written as
a sum of {\bf  positive  noncompact roots} (recall  our identification for which  $\Delta, \Delta_c \subset  \t^*$).
The number $\mathcal P_n(\gamma)$ is a well-defined integer, since the elements of $\Delta_n^+$  span a cone which
contains no straight lines. As usual, $\epsilon(w)$ will stand
for the sign of $w$. 
Remark that, as $\mu+\rho_c$ and  $\lambda+\rho_n+\rho_c$ are weights of $T$, the element  $w\mu-\lambda-\rho_n$ is a weight of $T$.

\noindent  It is convenient to extend the definition of $m_\mu^\lambda$ to an antisymmetric function on $P_\k^r.$   As we observed already an element $\mu\in P_\k^r$ can be conjugated to a unique  regular element in the corresponding  positive chamber $\a_c\subset \t^*$ for compact roots,
via an element  $w \in {\mathcal W}_c.$ Thus we define $m_{\mu}^\lambda=\epsilon(w)m_{w\mu}^\lambda.$  Of course with this generalization the multiplicity of the  K-type $\mu$  is  $| m_{\mu}^\lambda|$ and we can complete our picture for $U(2,1)$, Fig.\ref{figureu21}, in the following way:
\begin{figure}[h]
 \begin{center}
 \includegraphics[scale=0.8]{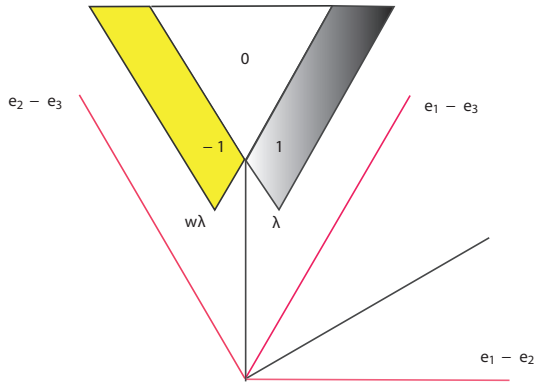}
\caption{$m_{\mu}^{\lambda}$ as antisymmetric function on $U(2,1)$. \label{figureu21}}

\end{center}
\end{figure}

\noindent The representation $\tau_{ lowest}$ with Harish-Chandra parameter
$\mu_{lowest}=\lambda+\rho_n$ is the lowest $K$-type of the representation $\pi^{\lambda}$ and occurs with multiplicity $1$. It is,
 in general, difficult to compute $m^\lambda_\mu$ for general $\mu$.


We will use Blattner's formula to compute $m^\lambda_\mu$.
 Our algorithm is based on a general scheme for computing partition functions using  multidimensional residues.
Note that  the presence of signs in Blattner formula doesn't even allow to say if a $K$ -type appears without fully computing  its multiplicity.

Recall that if $\CA^+$  is   a positive  root system of a semi-simple Lie algebra, the formula for the partition function has been used to compute  tensor product decomposition or  weight multiplicities (\cite {C1},\cite{C2}).

\subsection{Polynomial behavior of the Duistermaat-Heckman measure}

Recall Paradan's  results (\cite{par}) on the  behavior of the function $m_{\mu}^{\lambda}$ and its support.
For this it is useful to  first recall the semi-classical analog of Blattner formula.

Let us fix  a positive  system of roots $\Delta^+$ and consider the corresponding (closed) positive chambers $\a,\a_c\subset \t^*$ for $\Delta^+$ and  $\Delta_c^+$.
Our parameter $\lambda$ varies in $\a$ and it is non singular.
In this subsection, the integrality condition $\lambda\in P_{\g}^r$ is not required.

Let $O_\lambda\subset \g^*$ be the
 coadjoint orbit of $\lambda$.
It is a symplectic manifold, and thus is provided with a Liouville measure
$d\beta_\lambda$.
Let $p:O_\lambda\to \k^*$ be the projection. This is a proper map.
Each  coadjoint $K$-orbit in $\k^*$ intersect $\a_c\subset \t^*\subset \k^*$. Thus the projection of
 $O_\lambda$ on $\k^*$ is entirely determined by its intersection with $\a_c$.
We recall that the set $p(O_\lambda)\cap \a_c$ is a closed convex polyhedron.

\begin{definition}
The Kirwan polyhedron ${\rm Kirwan}(\lambda)$ is the polyhedron
$p(O_\lambda)\cap \a_c$.
\end{definition}

As far as we know, there are no algorithm to determine the Kirwan polyhedron.

A weak result on the  support of   ${\rm Kirwan}(\lambda)$ is that
${\rm Kirwan}(\lambda)$ is contained in $\lambda+{\rm Cone}(\Delta_n^+)$ where
${\rm Cone}(\Delta_n^+)$ is the cone generated by positive non compact roots.

The push-forward of the measure $d\beta_\lambda$  along the projection $p: O_\lambda\to \k^*$ gives us an invariant positive measure on $\k^*$.  By quotienting this measure by the signed Liouville measures $d\beta_{K\mu}$ of the coadjoint orbits $K\mu$  in $\k^*$, we obtain a ${\mathcal W}_c$-anti-invariant measure $dF^{\lambda}$ on $\t^*$. More precisely, for $\phi$ a test function on $\k^*$,
\begin{equation}\label{eq:DHa}
\int_{O_\lambda}  \phi(p(f)) d\beta_\lambda(f)   =\frac{1}{\# {\mathcal W}_c}\int_{ \t^*}
dF^\lambda(\mu)\epsilon_\mu \left(
 \int_{K\mu}
 d\beta_{K\mu}(f)\phi(f)\right).
\end{equation}
Here $\epsilon_\mu$ is the locally constant function on $\t^*$
 anti-invariant  by $ {\mathcal W}_c$ and equal to $1$ on the interior of the positive chamber $\a_c$.
We refer to $dF^{\lambda}$ as the Duistermaat-Heckman measure.

If $\CA^+=[\alpha_1,\ldots,\alpha_N]$ is a sequence of elements in $\t^*$ spanning a pointed cone,
 the multispline distribution  $Y_{ \mathcal A^+}$ is defined by the following formula.
For $\phi$ a test function on $\t^*$:

\begin{equation}\label{spline}
\langle Y_{\mathcal A^+},\phi\rangle=\int_{0}^{\infty}\cdots \int_0^{\infty} \phi(\sum_{i=1}^N t_i\alpha_i) dt_1\cdots dt_N.
\end{equation}

Then, for $\lambda\in \a$ and $\mu \in \t^*,$ we have the following result due to Duflo-Heckman-Vergne,\cite{dhv}:
\begin{equation} \label{eq:dhvpol}
dF^{\lambda}(\mu)=  \sum_{w\in
 {\mathcal W}_c} \epsilon(w) w(\delta_\lambda*Y_{\Delta_n^+}).
\end{equation}
where $\Delta^+_n$ is the system of positive noncompact roots defined by $\a.$

To simplify our next statements, assume that $G$ is semi-simple and has no compact factors. Then $\t^*$ is generated by non compact roots.
Recall that the spline function is given by a locally polynomial function on $\t^*$ well defined outside a finite number of hyperplanes (see the description later).
Choosing the Lebesgue measure $dh$ associated to the root lattice,
 we may identify the measure $Y_{\Delta_n^+}$ to a function, denoted by $Y_n^+$. Similarly, we identify the measure $dF^{\lambda}$ to a function $F^{\lambda}$ on $\t^*$.
Then we have,  almost everywhere, the semi-classical analogue of Blattner formula:

$$ F^{\lambda}(\mu)=
\sum_{w\in
 {\mathcal W}_c} \epsilon(w) Y^+_{n}(w\mu-\lambda)$$
 $\lambda\in \a$  and regular.

The (anti-invariant) function $F^{\lambda}(\mu)$  restricted to the positive compact chamber $\a_c$  is a non negative measure with  support the Kirwan polyhedron ${\rm Kirwan}(\lambda)$.

\bigskip

It follows  from the study of spline functions  that there exists a finite number of open polyhedral cones $R^i$ in $\a\times \a_c$ (so that the union of the cones  $\overline R^i$ cover $\a\times \a_c$)
and polynomial functions $p^i$  on $\a \times \a_c$ such that
$F^{\lambda}(\mu)$ is given, for  $\lambda\in \a,\mu \in \a_c, \ (\lambda,\mu) \in R^i,$
by the polynomial
 $p^i(\lambda,\mu)$ on $R^i(\lambda)=\{\mu\in \a_c,\  (\lambda,\mu)\in R^i\}.$
 In particular the Kirwan polyhedron ${\rm Kirwan}(\lambda)$ is the union of the regions
 $\overline{R^i(\lambda)}$ for which the polynomial $p^{i}$ restricted to
  $\overline{R^i(\lambda)}$  is not equal to $0$.
In fact the functions $p^i(\lambda,\mu) $ are linear combinations of   polynomial functions of $w\lambda-\mu$ where $w$ are some elements of  $\mathcal W_c$.

If $R$ is an open  cone in $\a\times \a_c$ such that $F^{\lambda}(\mu)$ is given by a polynomial formula $p^R(\lambda,\mu)$ when $(\lambda,\mu)\in R$, $\lambda \in \a$,  we say that $R$ is a domain of polynomiality and that $p^R$ is the local polynomial for $F^{\lambda}$ on $R$.

Let us finally recall that the local polynomials $p^{R}$ belong  to some particular space of polynomials
satisfying some system of partial differential equations.
For $\alpha$ a non compact root, consider the derivative $\partial_\alpha$.
We say that an hyperplane $H\in\t^*$ is admissible for $\Delta_n$ if $H$ is spanned by a subset of $\dim \t-1$ non compact roots, that is roots in $\Delta_n.$
We denote by $\mathcal H_n$ the set of admissible hyperplanes for $\Delta_n$.

\begin{definition}
A polynomial $p$ on $\t^*$ is in the Dahmen-Micchelli space $D(\Delta_n^+)$ if
$p$ satisfies the system of equations:
$$(\prod_{\alpha\in \Delta_n^+\setminus Q} \partial_\alpha) p=0$$
for any $Q\in \CH_n$.
\end{definition}

Remark that the space $D(\Delta_n^+)$ depends only of $\Delta_n$ and not of a choice of $\Delta_n^+$.

Then, the following result follows from Dahmen-Micchelli theory of the splines.

\begin{proposition}
For any domain of polynomiality $R$,
the polynomial $\mu \to p^{R}(\lambda,\mu)$ belongs to the space $D(\Delta_n^+).$
\end{proposition}

\subsection{Quasi-polynomiality results}

Let us come back to the discrete setting.
Let us fix  as before a positive  system of roots $\Delta^+$ and consider the corresponding chambers $\a,\a_c\subset \t^*$ for $\Delta^+$ and  $\Delta_c^+$.
Fix $\lambda\in P_\g^r\cap \a$ and $\mu\in P_\k^r\cap \a_c$.
We can then define $m^{\lambda}_\mu$, the multiplicity of  $\tau_\mu$ in the discrete series
$\pi^{\lambda}$.

By definition, a quasipolynomial function on a lattice $L$
 is a function on $L$ which coincides with
a polynomial on each coset of some sublattice $L'$ of finite index  in
$L$. The subsets $P_\g,P_\k$  are shifted lattices  and we may say that a function $k$ on $P_\k^r$ is quasi polynomial  on $P_\g\times P_\k$ if the shifted function $k(\lambda-\rho,\mu-\rho_c)$ is quasipolynomial on the lattice $P\times P$.

\begin{theorem}
Let
$R$ be a domain of polynomiality  in $\a\times \a_c$ for  the Duistermaat-Heckman measure.
  Then there exists a quasi polynomial function $P^R$ on $P_\g\times  P_\k$
  such that
 $m_\mu^\lambda= P^{R}(\lambda,\mu)$ for any
 $(\lambda,\mu)\in \overline{R}\cap (P_\g^r\times P_\k^r),$ \ $\lambda \in \a, \mu\in \a_c.$
\end{theorem}

(In fact the functions $P^R$ are linear combinations of  quasi polynomial functions of $w\lambda-\mu$ where $w$ are some elements of  $\mathcal W_c$.)

The $K$-types occurring with non zero multiplicity in $\pi^{\lambda}$ are such that $\mu$ is in the interior of the Kirwan polyhedron ${\rm Kirwan}(\lambda)$.
In particular the lowest $K$-type $\mu_{lowest}$ is in the interior of   ${\rm Kirwan}(\lambda)$.
In particular all the $K$-types occurring with non zero multiplicity in $\pi^{\lambda}$ are such that $\mu$ is in the interior of the cone $\lambda+{\rm Cone}(\Delta_n^+)$.
We believe they are contained in  the cone
$\mu_{lowest}+{\rm Cone}(\Delta_n^+)$, but we do not know if this assertion is true or not (by Vogan's theorem, they are contained in  $\mu_{lowest}+{\rm Cone}(\Delta^+)$) .

If $v \in \t^*$, we say that $v$ is an    {\bf asymptotic  direction}, \label{defas} if the line
$\mu_{\rm lowest}+t v$ is contained in ${\rm Kirwan}(\lambda)$ for all $t\geq 0$.
The set of asymptotic directions form a cone, which determines the wave-front set of $\pi^{\lambda}|_K$.

 For the holomorphic discrete series, the descrition of the cone of asymptotic diections is known. In fact if the lowest $K$-type of $\pi^{\lambda}$ is a one dimensional representation of $K$, the exact support of the function $m_\mu^{\lambda}$ has been determined  by  Schmid.

We will explain in  Section \ref{general}  how to compute  regions of polynomiality $R$ and the quasi-polynomial $P^R$.

The quasi polynomials $P^{R}$ are in some particular space of  quasi polynomials
satisfying some system of partial difference equations.
For $\alpha$ a non compact root, consider the difference operator $\nabla_\alpha$ acting on $\Z$ valued functions on $P_\k$
 by
 $$(\nabla_\alpha k)(\mu)=k(\mu)-k(\mu-\alpha).$$

\begin{definition}
A quasi polynomial $L$ on $P_\k$ is in the Dahmen-Micchelli space $DM(\Delta_n^+)$ if
$p$ satisfies the system of equations:
$$(\prod_{\alpha\in \Delta_n^+\setminus Q} \nabla_\alpha) L=0$$
for any $Q\in \CH_n$.
\end{definition}
Then, the following result follows from Dahmen-Micchelli theory of partition functions.
\begin{proposition}
The quasi polynomial $\mu \to P^{R}(\lambda,\mu)$ belongs to the space $DM(\Delta_n^+).$
\end{proposition}

\subsection{Aim of the algorithm: what can we do?}\label{aim}
Our algorithm addresses    the following  questions for $U(p,q)$. All of these questions will be analyzed in more details in  Sec.\ref{general}.
\subsubsection{Numeric}

We enter as input two  parameters $\lambda,\mu\in P_\g^r\times  P_\k^r$.
The output is the integer $m_\mu^{\lambda}$, see Sec.\ref{Bnum}.

\subsubsection{Regions of polynomiality}
The input is two parameters $\lambda_0,\mu_0\in P_\g^r\times  P_\k^r$.
Let $\a,\a_c$ be the chambers determined by $\lambda_0$ and $\mu_0$.
We also give two symbolic parameters $\lambda,\mu.$

Then the output  is a closed  cone  $R(\lambda_0,\mu_0) \subset \a\oplus \a_c$ described by linear inequations in $\lambda,\mu$, containing $(\lambda_0,\mu_0)$ and a quasi-polynomial $P$  in $(\lambda,\mu)$
such that $m_\mu^{\lambda}=P(\lambda,\mu)$ for any $(\lambda,\mu)\in R(\lambda_0,\mu_0)\cap
(P_\g^r\times P_\k^r)$.

We worked out part of this  program  for $U(p,q)$, but it is still not fully implemented.

In particular, for the moment, we are not able to  produce a cover of $\a\times \a_c$ by such regions.
The number of regions needed grows very fast with the rank. Furthermore, we are not able to decide when  we have  finished to cover    $\a\times \a_c.$

%
%
%

\subsubsection{Asymptotic directions}\label {asdir}

We implemented  (for $U(p,q)$) a simpler question which gives a test for asymptotic directions.

Let's  consider as input parameters $\lambda_0$ in $P_\g^r$ and a weight $\vec{v}\in \a_c$.
Let $\mu_0$ be  the lowest $K$-type of $\pi^{\lambda_0}$.
The line $t\mapsto (\lambda_0,\mu_0+ t\vec{v})$ cross  domains of polynomiality $R^i$ at a certain finite number of points $0\leq t_1<t_2<\cdots<t_s$. Let us define $t_0=0,t_{s+1}=\infty.$ Then we study the function $P(t)= m^{\lambda_0}_{\mu_0+tv}, \ t\in \N, t\geq 0.$

 We can find    polynomials $q_{[t_i,t_{i+1}]}$ of degree bounded by $pq-(p+q+1)$  such that
$P(t)=q_{[t_i,t_{i+1}]}(t)$ when $t_i\leq t\leq t_{i+1}$ for $i=0,\ldots, s$ and $t\in \N.$
In particular, the direction $\vec{v}$ belongs to the asymptotic cone  of the Kirwan polyhedron, if and only if our last polynomial $q_{[t_{s},\infty]}$
is non zero, see Sec.\ref{asymp}.

As we discussed in the first part, if the intervals are two small, these polynomials are not uniquely determined. However the last interval is infinite, and the last polynomial is well determined.  If this last data  is non zero, then $\vec{v}$ is in the wave front set of $\pi^{\lambda}$.  The reciproc is not entirely clear.
Indeed for a direction to be in the wave front set, it is sufficient to  be approached in the projective space by   lines $\R^+\mu_n$ with $\mu_n\in P_\k^r\cap \a_c$ such that the multiplicity  $m_{\mu_n}^{\lambda}$ is non zero, and  the sequence $\mu_n$ is  going to the infinity in $\a_c$.
 Thus  we do not know if a rational  line $\mu_0+t \vec{v}$ contained in the Kirwan polytope  could totally avoid the support of the function $m_{\mu}^{\lambda}$.
We do not think this is possible.

\section{Partition functions: the general scheme.}

 \subsection{Definitions}
 Let $U$ be a $r$-dimensional real vector space and
$V$ be its dual vector space. We fix the choice of a Lebesgue measure $dh$
on $V$. Consider a list   $\CA^+$ of non-zero generators for $V$  given by
$$\CA^+=[\alpha_1,\alpha_2,\ldots,\alpha_N].$$

We recall several results concerning partition functions that appear in \cite{BBCV} in the general context. However, let us describe  right away the system of vectors $\Delta^{+}(A,B)\subset A_r$  that  will appear in our programs and that describe   parabolic subsystem of $A_r$,
(as we said the same method could be applied to other parabolic root systems).

 \begin{example}\label{Ar}
 \begin{itemize}

\item  Let $E$ be an $r$-dimensional vector space with basis $e_i$  ($i=1$, \ldots, $r+1$). Consider the sequence
 $$A^+_r=[e_i-e_j\,|\,1\leq i<j\leq {r+1}].$$  This is a  system of  positive roots  of type $A_{r}$.
 We  let $V$ to be the vector  space
  $$V=\left\{h=\sum_{i=1}^{r+1} h_ie_i\in E\,\Big|\,  \sum_{i=1}^{r+1} h_i=0\right\}.$$
 Let $V_\Z$ be the lattice spanned by $A^+_r$.
We may identify $V$ with
$\R^r$  by $h\mapsto [h_1,h_2,\ldots, h_r]$.
In this identification the lattice  $V_\Z$  is identified with $\Z^r$.


\item Let $A,B$ be two complementary subsets of $[1,2,3,\ldots, p+q]$ with $|A|=p$ and $|B |=q$.
Let $r=p+q-1$.
We define $\Delta^+(A,B)$ as the sublist of $A_{r}^+$ defined by
$$\Delta^+(A,B)=[e_i-e_j\,|\,1\leq i<j\leq {r+1}], \text{with} \ \ i\in A, j\in B\, \text{or} \   i\in B, j\in A.$$
 These systems are  the system of  positive  roots  for the maximal parabolics of $\mathfrak g \mathfrak l (r+1)$, for different  choices of orders.
\end{itemize}

 \end{example}

Let us go back to the general scheme.

\noindent For any subset $S$ of $V$, we denote by
$\CC(S)$ the convex cone generated by non-negative linear
combinations of elements of $S$. We assume that the convex cone
$\CC(\CA^+)$ is acute in $V$ with non-empty interior.

\noindent If $S$ is a subset of $V$, we denote by $<S>$ the vector space spanned by $S$.
\medskip

\begin{definition} \label{admissible}
A hyperplane $H$ in $V$ is \emph{$\CA^+$-admissible} if it
is spanned by a set of vectors of $\CA^+$.

When $\CA^+=\Delta^+_n(\lambda)$ or $\Delta^+(A,B)$ then an $\CA^+$-admissible hyperplane will be also called   a \emph {noncompact wall}.
\end{definition}

\noindent
\textbf  {Chambers}

\noindent  Let $\CV_{sing}(\CA^+)$ be the union of the boundaries of the cones
$\CC(S)$, where $S$ ranges over all the subsets of $\CA^+$. The
complement of $\CV_{sing}(\CA^+)$ in $V$ is by definition the
open set $\CC_{reg}(\CA^+)$ of \emph{regular} elements. A connected
component $\gc$ of $\CC_{reg}(\CA^+)$ is called a \emph{chamber} of
$\CC(\CA^+)$.
Remark that, in our definition,
            the complement of the cone  $\CC(\CA^+)$  in $V$ is a chamber that we call the \emph{exterior} chamber.
The chambers contained in $\CC(\CA^+)$, that we will call \emph{ interior } chambers, are open convex cones.
Sometimes chambers are called cells or big cells by  other authors.

\noindent The faces of the interior chambers span admissible hyperplanes.

\noindent The following pictures illustrate the situation for the (interior) chambers in the case of $A_3^+,$ Fig.\ref {figurea3}, and   the (interior) chambers for various subsystems of   $A_3^+$, of type $\Delta^+(A,B)$ relative to $U(2,2)$, Fig.\ref{figurea3nonc}.

\begin{figure}[h!]
 \begin{center} \includegraphics[scale=0.9]{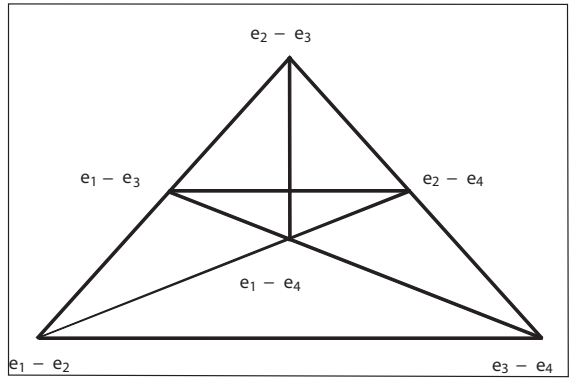}
\caption{The 7 chambers for $A^+_3$ \label{figurea3}}
\end{center}
\end{figure}
\newpage
\begin{figure}[h!]
 \centering
 \includegraphics[scale=0.8]{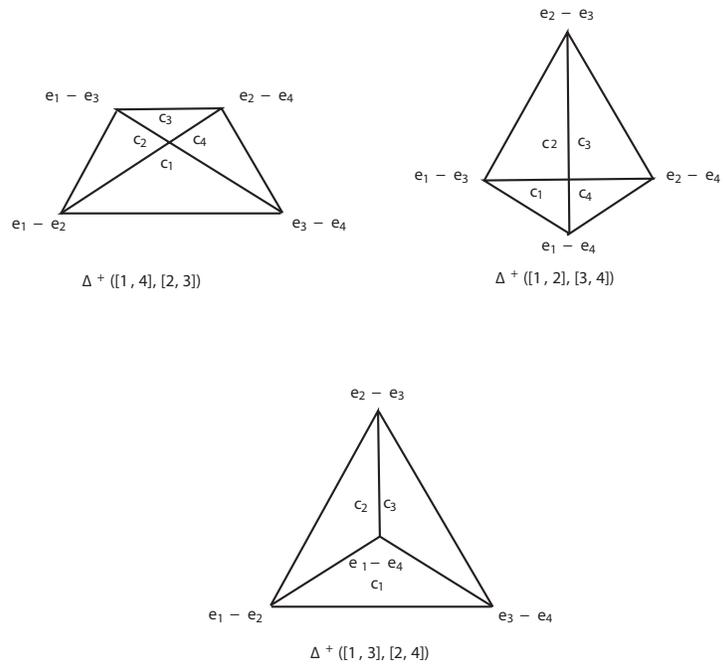}
\caption{Parabolic subsystems of   $A_3^+$ \label{figurea3nonc}}
\end{figure}

\newpage
\noindent\textbf  {Polytopes}

\noindent
We consider the space $\R^N$ with its standard basis $\omega_i$
and Lebesgue measure $dx$.
If $x=\sum_{i=1}^N x_i\omega_i\in\R^N$, we simply write $x=(x_1,\ldots,x_N).$
Consider the surjective map $A:\R^N\to V$ defined by
$A(\omega_i)=\alpha_i.$

\noindent If $h\in V$, we define the convex polytope $\Pi_{\CA^+}(h)$ consisting of all non-negative solutions of the system
of $r$ linear equations
$\sum_{i=1}^N x_i\alpha_i=h$
that is
$$\Pi_{\CA^+}(h)=\left\{x=(x_1,\ldots,x_N) \in\R^N\,|\,Ax=h, \ x_i\geq 0\right\}.$$
\medskip

We call   $\Pi_{\CA^+}(h)$ a partition polytope (associated to $\CA^+$ and $h$).

\bigskip

We identify the spline distribution $Y_{\mathcal A^+}$  (by Formula \ref{spline})
to a function  still denoted by $Y_{\mathcal A^+}$  using $dh$.

Recall the following theorem, which follows right away from Fubini theorem, (\cite{baldonivergne}, \cite{BDV}.)

\begin{theorem}\label{theo:volume}
The value of the spline function $Y_{\mathcal A^+}$   at $h$ is the volume of the partition polytope  $\Pi_{\CA^+}(h)$ for the quotient measure $dx/dh$.

\end{theorem}

The spline function   $Y_{\mathcal A^+}$  is given by a polynomial formula on each interior chamber. It is identically equal to $0$ on the exterior chamber.

\bigskip

\noindent\textbf  {Partition functions}

\noindent Let $V_\Z$  be a lattice in $V$  and suppose now  that the elements $\alpha_i$  of  our sequence   $\CA^+$  belong to the lattice $V_\Z$.
\noindent   If $h\in V_\Z$ we define
$N_{\CA^+}(h)=\vert\Pi_{\CA^+}(h)\cap \Z^N \vert,$
the number of integral points in the partition polytope
$\Pi_{\CA^+}(h).$

Thus $N_{\CA^+}(h)$ is the number of solutions $(x_1,x_2, \ldots, x_N)$,
in non-negative integers $x_j$, of the equation $\sum_{j=1}^N x_j
\alpha_j={h}.$

\noindent The function $h\mapsto N_{\CA^+}(h)$ is called the partition function
of $\CA^+$. We refer to it as   {\bf  Kostant partition function}.

We will see after stating Theorem~\ref{theo.main} that
$h\mapsto N_{\CA^+}(h)$ is
quasipolynomial on each chamber.

Let us recall briefly the theory that allows to compute
Kostant partition functions.

\medskip
\noindent\textbf  {Jeffrey-Kirwan residue}

 Let $\nu$ be a subset of $\{1,2,\ldots,N\}$.
We will say that $\nu$ is \emph{generating} (respectively
\emph{basic}) if the set $\{\alpha_i\,|\,i\in\nu\}$ generates
(respectively is a basis of) the vector space $V$.
We write
$\bases(\CA^+)$ for the set of basic subsets.

Let  $\CR_{\CA^+}$ be the ring of rational
functions  on $U$, the dual vector space to $V$, with poles on hyperplanes determined by kernel of elements $\alpha \in \CA^+$.

  $\CR_{\CA^+}$
is $\Z$-graded by degree. Every function in   $\CR_{\CA^+}$
 of degree $-r$  decomposes (see \cite{BriVer97})  as the sum  of basic fractions $f_\sigma,$
 $f_\sigma=\frac{1}{\prod_{i\in\sigma}\alpha_i}, \ \sigma\in\bases(\CA^+)
$
and degenerate fractions; here degenerate fractions are those for which
the linear forms in the denominator do not span $V$.

Now having
fixed a chamber $\gc$, we define a functional $\JK_{\gc}(f_{\sigma})$ on
$\CR_{\CA^+}$ called the Jeffrey-Kirwan residue (or
JK \emph{  residue}) as follows:
\begin{equation} \label{defi.jk}
\JK_{\gc}(f_\sigma)=
\begin{cases}
  \vol(\sigma)^{-1},&\mbox{if }\gc\subset\CC(\sigma),\\
  0, &\mbox{if }\gc\cap\CC(\sigma)=\emptyset\
\end{cases}
\end{equation}
where  $\sigma\in\bases(\CA^+)$ and $\vol(\sigma)$ is  the volume of the parallelotope
$\sum_{i=1}^r[0,1]\alpha_i$ computed for the measure $dh$.

There exists a linear form $\JK_{\gc}$, that we call the Jeffrey-Kirwan  residue, on $\CR_{\CA^+}$   such that $\JK_{\gc}$ takes the above values on the elements $f_\sigma$, and  is equal to $0$ on  a degenerate fraction or
on  a rational function of pure degree different from $-r$.

If $\gc$ is the exterior chamber, then clearly $\JK_\gc$ is equal to $0$, as $\gc$ is not contained in $\CC(\CA^+)$.

We may go further and extend the definition of the Jeffrey-Kirwan residue  to the space
${\widehat\CR}_{\CA}$ which is the space consisting of functions
$P/Q$ where $Q$ is a product of powers of the linear forms
$\alpha_i$ and $P=\sum_{k=0}^{\infty}P_k$ is a formal power
series. Then we just define, if $Q$ is of degree $q$,
$$\JK_{\gc}(P/Q)=\JK_{\gc}(P_{q-r}/Q)$$
as the JK residue of the component of degree $-r$ of $P/Q$.

\subsection {Spline functions and Kostant partition function}

Let us  recall  the formulae for  the spline function $Y_{\mathcal A^+}$
and for
$N_{\CA^+}(h)$.

\begin{definition}
Let $\gc$ be an  chamber contained in the cone $\CC(\CA^+)$.
Define the function ${\bf Y}^{\gc}$ on $V$ by

$${\bf Y}_{\CA^+}^{\gc}(h)= \JK_{\gc}\left(\frac{e^{ h}}
                       {\prod_{i=1}^N\alpha_i}\right).$$

\end{definition}

More explicitly, as $\JK_{\gc}$ vanishes outside the degree $-r$, we have

$${\bf Y}_{\CA^+}^{\gc}(h)
    =\frac{1}{(N-r)!} \JK_{\gc}\left(\frac{ h^{N-r}}
                       {\prod_{i=1}^N\alpha_i}\right).$$

We thus see  ${\bf Y}_{\CA^+}^{\gc}(h)$ is an homogeneous  polynomial on $V$.

The proof of the  following theorem is immediate (\cite{baldonivergne}).
\begin{theorem} \label{theo.main1}
Let $Y_{\CA^+}(h)$ be the multispline function associated to $\CA^+$.
    Let $\gc$ be a  chamber contained in the cone $\CC(\CA^+)$.

 We have for $h\in \gc$:

$$Y_{\CA^+}(h)={\bf Y}_{\CA^+}^{\gc}(h).$$

\end{theorem}

\begin{remark}

According to Theorem \ref{theo:volume}, this theorem gives the formula for the volume $V_{\CA^+}(h)$ of the partition  polytope $\Pi_{\CA^+}(h)$.

\end{remark}

Let us now give the residue formula for the number of integral points $N_{\CA^+}$
 of the partition polytope $\Pi_{\CA^+}(h)$.

Consider the torus
$T=U/U_\Z$  where $U$ is the dual vector space to $V$ and  $U_\Z\subset U$  is the dual lattice to $V_\Z$.
If $G\in U$, we denote by $g$ its image in $T$.

For $\sigma\in\bases(\CA^+)$ we
denote by $T(\sigma)$ the  subset of $T$ defined by

$$T(\sigma)=\left\{g\in T\,\Big|\,\,
  e^{\ll\alpha,2\pi\sqrt{-1}G\rr}=1\,\,
  \mbox{for all}\,\alpha\in\sigma,\  \ G\  \mbox  {a \ representative \ of  } g\in U/U_\Z\right \}.$$

  The set $T(\sigma)$ is a finite subset of $T$.

For $G\in U$ and $h\in V$, consider the {\bf Kostant function}
$K(G,h)$ on $U$ defined by
\begin{equation}\label{equa.Fgh}
K(G,h)(u)=\frac{e^{\ll h,2\pi\sqrt{-1}G+u\rr}}
     {\prod_{i=1}^N(1-e^{-\ll\alpha_i,2\pi\sqrt{-1}G+u\rr})}.
\end{equation}

\begin{remark}
If $h\in V_\Z$, the function   $K(G,h)$ depends only of the class $g$ of $G$ in $U/U_\Z$.
\end{remark}

\noindent The function $K(G,h)(u)$ is an element of ${\widehat\CR}_{\CA}$.

\noindent Indeed if we write $I(g)=\left\{i\,\Big|\,\,1\leq i\leq N,
  e^{-\ll\alpha_i,2\pi\sqrt{-1}G\rr}=1\right\},$
then
\begin{equation} \label{equa.fpsi}
K(G,h)(u)
  =e^{\ll h,2\pi\sqrt{-1}G\rr}\frac{e^{\ll h,u\rr} \psi^g(u)}
        {\prod_{i\in I(g)}\ll\alpha_i,u\rr}
        \end{equation}
where $\psi^g(u)$ is the holomorphic function of $u$ (in a
neighborhood of zero) defined by
$$\psi^g(u)=
  \prod_{i\in I(g)}
    \frac{\ll\alpha_i,u\rr}
         {(1-e^{-\ll\alpha_i,u\rr})}
  \times
  \prod_{i\notin I(g)}
    \frac{1}
         {(1-e^{-\ll\alpha_i,2\pi\sqrt{-1}G+u\rr})}.$$
By taking the Taylor series of $e^{\langle h,u\rangle }\psi^g(u)$ at $u=0$, we see that the function  $u\to K(G,h)(u)$ on $U$  defines an element of
 ${\widehat\CR}_{\CA}$.
If $\gc$ is a chamber of $\CC(\CA^+)$, the Jeffrey-Kirwan
residue $\JK_\gc(K(g,h))$ is  thus well defined.

\begin{definition}\label{def:NKc}
Let $\gc$ be a chamber.
Let $F$ be a finite subset of $U$. We define the function   ${\bf N}_{\CA^+}^{\gc,F}$ on $V$ by

  $${\bf N}_{\CA^+}^{\gc,F}(h)=\vol(V/V_\Z,dh)\sum_{G\in F}\JK_{\gc}(K(G,h))$$
  where  $\vol(V/V_\Z,dh)$ is the volume of the fundamental domain of $V_{\Z}$ for $dh$.

\end{definition}

Finally introduce the zonotope $Z(\CA^+)$ to be the convex polyhedra defined by
$$Z(\CA^+):=\{\sum_{i=1}^N t_i \alpha_i; 0\leq t_i\leq 1\}.$$
When $\CA^+$ is fixed,
we just write $Z=Z(\CA^+)$, and if $C$ is a set, we denote by   $C-Z$  the set of elements $\{\xi-z\}$ where $\xi\in C$ and $z\in Z$.

The following theorem is due to Szenes-Vergne~\cite{SzeVer}. It generalizes  \cite{DM}, \cite{KhoPuk} and \cite{BriVer97}.

\begin{theorem} \label{theo.main}
Let $\gc$ be a chamber.
Let $F$ be a finite subset of $U$. Assume that
  for any
  $\sigma\in\bases(\CA^+)$ such that  $\gc\subset \CC(\sigma)$, we have $T(\sigma)\subset F/U_\Z$.

Then for $h\in V_\Z\cap ({\gc-Z})$, we have
$$N_{\CA^+}(h)={\bf N}_{\CA^+}^{\gc,F}(h).$$
\end{theorem}

We choose for any chamber $\gc$ such a finite set $F$ such that all elements $g\in F/U_\Z$ have finite order
and such that $F$  satisfies the condition:

{(\bf C}) \ \label{conditionC}
 \ \  for any
  $\sigma\in\bases(\CA^+)$ such that  $\gc\subset \CC(\sigma)$, we have $T(\sigma)\subset F/U_\Z.$

It is possible to achieve this, for example choosing  a set $F$   of representatives of $\frac{1}{p}U_\Z$ modulo $U_\Z$, where $p$ is that $pU_\Z$ is contained in $\sum_{i\in \sigma} \Z \alpha_i$ for any basis $\sigma$.

We now simply denote ${\bf N}^{\c,F}$ by ${\bf N}^{\c}$, leaving implicit  the choice of the finite set $F$.
\begin{remark}

$\bullet$
Observe that  ${\bf N}_{\CA^+}^{\gc}(h)$  does
not depend on the measure $dh$, as it should be.

$\bullet$      If  $\gc$ is the exterior chamber, then ${\bf N}_{\CA^+}^{\gc}(h)=0$.
In our algorithm, we are not knowing in advance if the point $h$ belongs to the cone $\CC(\CA^+)$ or not, so that this remark is not as stupid as it looks.

$\bullet$ Observe also that if $\gc$ is an interior chamber, then $\gc-Z$ contains the closure $\overline \gc$ of $\gc$, while if $\gc$ is the exterior chamber $\gc-Z=\gc$.
For an interior chamber, usually the set  $\gc-Z$  intersected with the lattice $V_\Z$ is
strictly larger than $\overline \gc$ intersected with $V_\Z$.
This fact will be important for computing shifted partition functions, as we will explain later.

\end{remark}

Let us explain the behavior of the partition function $N_{\CA^+}$  on the domain  $\gc-Z$.

We first explain the case of an unimodular system.

\begin{definition}The system
  $\CA^+$  is unimodular  if each $\sigma \in
\bases(\CA^+)$ is a $\Z$-basis of $V_\Z$.
\end{definition}

\begin{example}
It is easy to see that $A_r^+$ is unimodular, so is any subsystem.
\end{example}

Thus if $\CA^+$ is unimodular, the set  $F=\{0\}$  satisfies the condition ({\bf C})   and we choose this set $F$.

\begin{proposition}

If $\CA^+$ is unimodular,
the function ${\bf N}^{\gc}_{\CA^+}(h)$ is a polynomial function on $V$.
\end{proposition}

\begin{proof}
We have just to consider  $K(G,h)=K(0,h)$ and
we can write
$$K(0,h)(u)=\frac{e^{\ll h,u\rr}}
       {\prod_{i=1}^N(1-e^{-\ll\alpha_i,u\rr})}
= \frac{e^{\ll h,u\rr}}
       {\prod_{i=1}^N\ll\alpha_i,u\rr}
  \times
  \frac{\prod_{i=1}^N\ll\alpha_i,u\rr}
       {\prod_{i=1}^N(1-e^{-\ll\alpha_i,u\rr})}$$
where $\frac{\prod_{i=1}^N\ll\alpha_i,u\rr}
       {\prod_{i=1}^N(1-e^{-\ll\alpha_i,u\rr})}
= \sum_{k=0}^{+\infty}\psi_k(u)$
is a holomorphic function of $u$ in a neighborhood of $0$ with
$\psi_0(u)=1$.

It follows that
${\bf N}^{\gc}_{\CA^+}(h)$  is given  by the following polynomial function   of $h$

   \begin{eqnarray}{\bf N}_{\CA^+}^{\gc}(h) &=&\vol \left( V/V_\Z, dh \right) \JK_{\gc}\left(
           \frac{e^{\ll h,u\rr}}
                {\prod_{i=1}^N\ll\alpha_i,u\rr}
           \times
           \sum_{k=0}^{+\infty}\psi_k(u)
           \right) \nonumber \\
&=&\vol \left( V/V_\Z, dh \right) \sum_{k=0}^{N-r}
   \frac{1}{(N-r-k)!}
   \JK_{\gc}\left(
           \frac{\ll h,u\rr^{N-r-k}\psi_k(u)}
                {\prod_{i=1}^N\ll\alpha_i,u\rr}
           \right). \label{na+formula3}
\end{eqnarray}

Note that the function ${\bf N}_{\CA^+}^{\gc}$ is a polynomial function of degree $N-r$ whose homogeneous
component of degree $N-r$ is the  function ${\bf Y}^{\gc}_{\CA^+}(h),$ that is the volume of the polytope.

\end{proof}

\bigskip

Let us now consider the general case where $F$ is no longer reduced to $\{0\}$. For example for  parabolic  root
systems  of $B_r$, $C_r$, $D_r$, the set $F$  satisfying the condition ({\bf C})  cannot  longer be taken as equal to
$\{0\}$.

We recall that an exponential polynomial function is a linear combination of exponential functions multiplied by polynomials.
\begin{proposition}
The function
 ${\bf N}^{\gc}_{\CA^+}(h)$  is an exponential polynomial function on $V$
   and the restriction of    ${\bf N}^{\gc}_{\CA^+}(h)$ to $V_\Z$
is a quasipolynomial function on $V_\Z$.

\end{proposition}

\begin{proof}
           Let us denote by
$\psi^g(u)=\sum_{k=0}^{+\infty}\psi_k^g(u)$ the series development
of the holomorphic function $\psi^g$ appearing in
formula~\eqref{equa.fpsi}. Then we see that $\JK_{\gc}(K(G,h))$
equals
\begin{eqnarray}
& &
       \left(e^{\ll h,2\pi\sqrt{-1}G\rr} JK_{\gc}\frac{e^{\ll h,u\rr}}
                  {\prod_{i\in I(g)}\ll\alpha_i,u\rr}
             \psi^g(u)
       \right) \label{na+formula1} \\
&=&e^{\ll h,2\pi\sqrt{-1}G\rr}\sum_{k=0}^{|I(g)|-r}
      \frac{1}{(|I(g)|-r-k)!}
      \JK_{\gc}\left(\frac{\ll h,u\rr^{|I(g)|-r-k}}
                         {\prod_{i\in I(g)}\ll\alpha_i,u\rr}
                    \psi_k^g(u)
              \right). \nonumber
\end{eqnarray}

The function $$h\mapsto
\JK_{\gc}\left(\frac{\ll h,u\rr^{|I(g)|-r-k}}
                         {\prod_{i\in I(g)}\ll\alpha_i,u\rr}
                    \psi_k^g(u)
              \right)$$
 is a polynomial function of $h$ of degree $|I(g)|-r-k$.
 Thus  we see that  $\JK_{\gc}(K(G,h))$
is the product of the exponential function  $e^{\ll h,2\pi\sqrt{-1}G\rr}$ by a polynomial function of $h$.

Furthermore, if $g$ is of order $p$ and $h$ varies in $V_\Z$, the function $h\mapsto e^{\ll
h,2\pi\sqrt{-1}G\rr}$ is constant on each coset $h+pV_\Z$
 of the
lattice $pV_\Z$.

\end{proof}

Return to the computation of the partition function $N_{\CA^+}(h)$.
Thus we see that when
$h$ varies in $(\gc-Z)\cap V_\Z$, we have  that $N_{\CA^+}(h)$ coincide with the quasi polynomial function  ${\bf N}_{\CA^+}^{\gc}(h)$ above. Note that its highest degree component is polynomial and is again the  function ${\bf Y}^{\c}_{\CA^+}(h)$, the volume of the polytope $\Pi_{\CA^+}(h)$.

The quasipolynomial nature  of the integral-point counting
functions $N_\CA^+$ stems precisely from the root of unity in formula
\eqref{na+formula1}.

Furthermore   for  parabolic root
systems of type $B$, $C$, and $D$, these roots of unity are of
order 2, as in the following example. Thus we summarize the properties of our partition functions in the following remark:
\begin{remark}
\begin{itemize}
\item $\CA_r^+$ is unimodular, that is we can choose $F=0$ in Theorem \ref{theo.main}, and thus the partition function $N_{\Phi}$  for any subset $\Phi$ of $\CA^+_r$  coincide with a polynomial function on each domain $\gc-Z(\Phi).$
\item The integral-point counting functions $N_{ \Phi}$  for any subsystem of $B_r,C_r,D_r$ coincide with  quasipolynomials with period $2$ on each domain   $\gc-Z(\Phi).$
\end{itemize}

\end{remark}

We now compute the number of integral points in two different situations: a non unimodular case and a  unimodular one.
We treat the non unimodular case first.

\begin{example} Here $V$ is a vector space with real
coordinates and basis $e_1,e_2$ and $U=V^*$ has dual basis $e^1,e^2$. We write $v=\sum_{i=1}^2 v_i e_i\in V$  and $u=\sum_{i=1}^2h_ie^i\in U$  for elements in $V$ and $U$ respectively.
Let us compute the  number of integral points
for the positive non compact root system occuring for the holomorphic discrete series of $SO(5,\C):$  that is we fix
$\Delta^+:=\{e_1,e_2,e_1+e_2,e_1-e_2\}$
and $\CA^+=\Delta_n^+:=\{e_1,e_1+e_2,e_1-e_2\}.$ We also write  a vector $h=h_1e_1+h_2e_2$  in the cone $\CC(\CA^+)$ as $(h_1,h_2)$.
Of course, the calculation can be done by hand, but we illustrate the method in this very simple example.

Observe that the root lattice is $\Z e_1 \oplus \Z e_2 $ and
$\vol \left( V/V_\Z, dh \right) = 1$ for the measure $dh=dh_1 dh_2$.

There are  two chambers, namely
$\gc_1=\CC(\{e_1+e_2,e_1\})$ and
$\gc_2=\CC(\{e_1,e_1-e_2\})$.

 Now let us
compute the Jeffrey-Kirwan residues on the chambers.

 We have for example:

$$\begin{array}{rclrcl}
\JK_{\gc_1}\left(\frac{1}{u_1(u_1+u_2)}\right)&=&1,&
\JK_{\gc_2}\left(\frac{1}{u_1(u_1+u_2)}\right)&=&0,\\
\JK_{\gc_1}\left(\frac{1}{(u_1+u_2)(u_1-u_2)}\right)&=&\frac12,&
\JK_{\gc_2}\left(\frac{1}{(u_1+u_2)(u_1-u_2)}\right)&=&\frac12,\\
\JK_{\gc_1}\left(\frac{1}{u_1(u_1-u_2)}\right)&=&0,&
\JK_{\gc_2}\left(\frac{1}{u_1(u_1-u_2)}\right)&=&1\\
\end{array}$$

  For the number of integral points, we first note that
$F=\{(0,0),(1/2,1/2)\}$.
Consequently $N_{\Delta_n^+}(h)$ is equal to the Jeffrey-Kirwan residue
of $f_1=K((1,1),h)$ plus $f_2=K((1/2,1/2),h)$.
We rewrite the series $f_j$ ($j=1$, $2$) as
$f_j=f'_j\times e^{u_1h_1+u_2h_2}/u_1(u_1+u_2)(u_1-u_2)$
where
{\scriptsize
\begin{eqnarray*}
f'_1&=&\frac{u_1}{1-e^{-u_1}}
       \times
       \frac{u_1+u_2}{1-e^{-(u_1+u_2)}}
       \times
       \frac{u_1-u_2}{1-e^{-(u_1-u_2)}},\\
f'_2&=&\frac{u_1}{1+e^{-u_1}}
       \frac{u_1+u_2}{1-e^{-(u_1+u_2)}}
       \times
       \frac{u_1-u_2}{1-e^{-(u_1-u_2)}}\times(-1)^{h_1+h_2}.
\end{eqnarray*} }
Using the series expansions
$\frac{x}{1-e^{-x}}=1+\frac12x+\frac{1}{12}x^2+O(x^3)$ and
$\frac{x}{1+e^{-x}}=\frac12x+O(x^2)$, we obtain that the number of
integral points is the JK residue of
{\scriptsize
\begin{eqnarray*}
&&\frac { h_1+\frac12}{(u_1-u_2)(u_1+u_2)}+\frac {\frac12}{u_1(u_1+u_2)}+\frac {\frac12}{u_1(u_1-u_2)}
+\frac{h_2u_2}{u_1(u_1-u_2)(u_1+u_2)}
  +\frac{\frac12(-1)^{h_1+h_2}}{(u_1+u_2)(u_1-u_2)}\\
&&=\frac{h_1+\frac12+\frac12(-1)^{h_1+h_2}}{(u_1-u_2)(u_1+u_2)}+\frac {\frac12}{u_1(u_1+u_2)}+\frac {\frac12}{u_1(u_1-u_2)}
-\frac{ h_2}{u_1(u_1+u_2)}+\frac{ h_2}{(u_1+u_2)(u_1-u_2)}\\
\end{eqnarray*}}
We then obtain:
{\scriptsize
\begin{eqnarray*}
N_{\Delta_n^+}(h)&=& \frac12 h_1+\frac14 (-1)^{h_1+h_2}+\frac34-\frac12h_2,   \quad\mbox{ if }h\in\gc_1,\\
N_{\Delta_n^+}(h)&=&\frac12h_1+\frac14(-1)^{h_1+h_2}+\frac34+\frac12 h_2,    \quad\mbox{ if }h\in\gc_2,\\
\end{eqnarray*}}
Note that the functions $N_{\Delta^+_n}$ agree on walls, that is $h_2=0$,  and the
formulae above are valid on the closures of the chambers.
\end{example}

The second example treats the  unimodular case of $A_r^+$, see Example \ref{Ar}.
Since we have identified $V$ with $\R^r$, then we have a canonical identification of
$U=V^*$ with $\R^r$ defined by duality: $u\in\R^{r}$ to
$u=\sum_{i=1}^{r}u_i e^i\in E^*$, where $e^i$ is the dual
basis to $e_i$.
Thus the root $e_i-e_j$ ($1\leq i<j\leq r$) produces the linear
function $u_i-u_j$ on $U$, while the root $e_i-e_{r+1}$ produces
the linear function $u_i$.
Recall also the identification  $h=\sum_{i=1}^{r+1}h_i e_i=[h_1,\ldots,h_r]$,

We  compute the number of integral points for  the parabolic subsystems of  $U(2,2)$  illustrated in Fig.\ref{figurea3nonc}.

\begin{example}  We consider the 3 different systems of non compact roots as described in Fig.\ref{figurea3nonc}
and  give the formulae for the partition function.
\begin{enumerate} \item If   $\Delta_n^+=\Delta^+([1,4],[2,3])$  then
{\scriptsize
\[
N_{\Delta_n^+}(h)=
\left\{
\begin{array}{l@{\quad\mbox{if}\quad }l}
h_1+h_2+1&h\in\gc_1,\\
 h_1+ h_2+ h_3+1&h\in\gc_2,\\
 h_1+h_3+1&h\in\gc_3,\\
 h_1+1&h\in\gc_4\\
 \end{array}
\right.
\]}
\item If   $\Delta_n^+=\Delta^+([1,2],[3,4])$  then
{\scriptsize
\[
N_{\Delta_n^+}(h)=
\left\{
\begin{array}{l@{\quad\mbox{if}\quad }l}
1+ h_2&h\in\gc_1,\\
 1+h_1+ h_2+ h_3&h\in\gc_2,\\
 1+h_1&h\in\gc_3,\\
1-h_3&h\in\gc_4\\
\end{array}
\right.
\]}
\item If   $\Delta_n^+=\Delta^+([1,3],[2,4])$  then
{\scriptsize
\[
N_{\Delta_n^+}(h)=
\left\{
\begin{array}{l@{\quad\mbox{if}\quad }l}
1+h_1+h_2& h\in\gc_1\\
1+h_1 +h_2+h_3 &   h\in\gc_2,\\
 1+h_1& h\in\gc_3,\\

\end{array}
 \right.
 \]}
\end{enumerate}
We have to compute the Jeffrey-Kirwan residue of the function

$f=f_1 \times \frac{1}{\prod_{\alpha \in \Delta^+_n} \alpha}$ where $f_1(h)(u)=\prod_{\alpha \in \Delta^+_n} \frac{\la\alpha,u\ra}{1-e^{-\la\alpha,u\ra}}\times  e^{u_1h_1+u_2h_2+u_3h_3}.$
The computation is immediate since we need only term of degree one for the expansion of   $f_1$.  We omit the details. Remark though that once again the formulae agree on walls as it should be.
\end{example}

\subsection{Shifted partition functions}

Let us consider as before  our lattice $V_\Z$ and our sequence  $\CA^+$  of elements of $V_\Z$.
Let $$\rho_n=\frac{1}{2}\sum_{\alpha\in \CA^+} \alpha.$$

We introduce $$P_n=\rho_n+V_\Z.$$ Thus for any $\mu\in P_n$, the function $N_{\CA^+}(\mu-\rho_n)$ is well defined.

Let $\CH$ be the complement of all admissible hyperplanes, that is hyperplanes generated by elements of $\CA^+$, Def.\ref{admissible}.

\begin{definition}

A tope is a connected component of the open subset $V-\CH$ of $V$.
\end{definition}

We choose once for all a finite set $F$ of elements $G$ of $U$, so that the image of elements $g$  cover all groups $T(\sigma)$.

If $\tau$ is a tope, then $\tau$ is contained in a unique chamber $\gc$, and we  denote by ${\bf N}^\tau_{\CA^+}$
 the  exponential polynomial function  ${\bf N}^{\gc,F}_{\CA^+}$ given in Definition \ref{def:NKc}.
 If $\tau$ is not contained in $\CC(\CA^+)$, then ${\bf N}^{\tau}_{\CA^+}=0$.

The closures of the topes $\tau$ form a cover  of $V$.
A consequence of Theorem \ref{theo.main}, is the following.
\begin{theorem}\label{theoremtope}
For any tope $\tau$ such that $\mu\in \overline{\tau}\cap P_n$,
we have
$$N_{\CA^+}(\mu-\rho_n)={\bf N}_{\CA^+}^{\tau}(\mu-\rho_n).$$
\end{theorem}

\subsection{A formula for the Jeffrey-Kirwan residue } \label{sect.JK}
Having stated a formula for partition functions (or shifted partition functions) in terms of $\JK_{\gc}$, we will explicit it
 using the notion
of maximal proper nested sets, as developed in
\cite{DCP}, and the notion of iterated residues.
 The  algorithmic implementation of this formula is working in
a quite impressive way, at least for low dimension.

  This general scheme   will be then  be applied  to  Blattners' formula.

\subsubsection{Iterated residue}
If $f$ is a meromorphic function of one variable $z$ with a pole
of order less than or equal to $k$ at $z=0$, then we can write
$f(z)=Q(z)/z^k$, where $Q(z)$ is a holomorphic function near $z=0$.
If the Taylor series of $Q$ is given by
$Q(z)=\sum_{s=0}^{\infty}q_s z^s$, then as usual the residue at
$z=0$ of the function $f(z)=\sum_{s=0}^{\infty}q_s z^{s-k}$ is
the coefficient of $1/z$, that is, $q_{k-1}$.
We will denote it by $\res_{z=0}f(z)$.
To compute this residue we can either expand $Q$ into a power series
and search for the coefficient of $z^{-1}$, or employ the formula
\begin{equation} \label{equa.resf}
\res_{z=0}f(z)
=\frac{1}{(k-1)!}(\partial_z)^{k-1}\left(z^kf(z)\right)\Big|_{z=0}.
\end{equation}

We now introduce the notion of iterated residue on the space
 ${\CR}_{\CA^+}$.

Let $\vec{\nu}=[\alpha_1,\alpha_2,\ldots,\alpha_r]$ be an ordered
basis of $V$ consisting of elements of $\CA^+$  (here we have
implicitly renumbered the elements of $\CA^+$  in order that the
elements of our basis are listed first). We choose a system of
coordinates on $U$ such that $\alpha_i(u)=u_i$. A function
$\phi\in \CR_{\CA^+}$ is thus written as a  rational fraction
$\phi(u_1,u_2,\ldots,u_r)
  =\frac{P(u_1,u_2,\ldots,u_r)}{Q(u_1,u_2,\ldots,u_r)}$
where the denominator $Q$ is a product of linear forms.

\begin{definition} \label{defi.ires}
If $\phi\in\CR_{\CA}$, the iterated residue
$\Ires_{\vec{\nu}}(\phi)$ of $\phi$ for $\vec\nu$ is the scalar
$$\Ires_{\vec{\nu}}(\phi)
  =\res_{u_r=0}\res_{u_{r-1}=0}
   \cdots\res_{u_1=0}\phi(u_1,u_2,\ldots,u_r)$$
where each residue is taken assuming that the variables
with higher indices are considered constants.
\end{definition}

Keep in mind that at each step the residue operation augments the
homogeneous degree of a rational function  by $+1$ (as for example
$\res_{x=0}(1/xy)=1/y$) so that the iterated residue vanishes on
homogeneous elements $\phi\in \CR_{\CA}$, if the homogeneous
degree of $\phi$ is different from $-r$.

Observe that the value of $\Ires_{\vec{\nu}}(\phi)$ depends on the
order of ${\vec{\nu}}$. For example, for $f=1/(x(y-x))$ we have
$\res_{x=0}\res_{y=0}(f)=0$ and $\res_{y=0}\res_{x=0}(f)=1$.

\begin{remark} \label{rema.wedg}
Choose any basis $\gamma_1$, $\gamma_2$, \ldots, $\gamma_r$ of $V$
such that $\oplus_{k=1}^j\alpha_j=\oplus_{k=1}^j\gamma_j$ for
every $1\leq j\leq r$  and such that
$\gamma_1\wedge\gamma_2\wedge\cdots\wedge\gamma_r
  =\alpha_1\wedge\alpha_2\wedge\cdots\wedge\alpha_r$.
Then, by induction, it is easy to see that for $\phi\in \CR_{\CA^+}$

$$\res_{\alpha_r=0}\cdots\res_{\alpha_1=0}\phi
=\res_{\gamma_r=0}\cdots\res_{\gamma_1=0}\phi.$$

Thus given an ordered basis, we may modify $\alpha_2$ by
$\alpha_2+c\alpha_1$, \ldots, with the purpose of getting easier
computations.
\end{remark}
As for the usual residue, the iterated residue  can be
expressed as an integral as explained in \cite {BBCV}.
This fact allows change of variables.
\subsubsection{Maximal proper nested sets adapted to a vector}
We recall briefly the notion of maximal proper nested set, $MNPS$ in short,  and some of their properties (see \cite{DCP}).

A subset
$S$ of $\CA^+$  is \emph{complete} if $S=\ll S\rr\cap\CA^+:$  here recall that $\ll S\rr$  is the vector
space spanned by $S$.   A complete subset $S$ is called \emph{reducible} if we can find a
decomposition $V=V_1\oplus V_2$  such that $S=S_1\cup S_2$ with
$S_1\subset V_1$ and $S_2\subset V_2$. Otherwise $S$ is said to be
\emph{irreducible}.

 A set
$M=\{I_1,I_2,\ldots,I_k\}$ of irreducible subsets of $\CA^+$  is
called \emph{nested} if, given any subfamily $\{I_1,\ldots,I_m\}$ of $M$
such that there exists no $i$, $j$ with $I_i\subset I_j$, then the
set $I_1\cup\cdots\cup I_m$ is \emph {complete} and the elements
$I_j$ are the irreducible components of $I_1\cup I_2\cup\cdots\cup
I_m$.
Then every maximal nested set $M$, $MNS$ in short,  contains $\CA^+$ and has exactly $r$ elements.

We now recall how to construct all maximal nested sets.
We may assume  that $\CA^+$  is irreducible, otherwise just take one of the irreducible components.
If $M$ is a maximal nested set, the vector space  $\ll M\setminus \CA^+\rr $
is an hyperplane $H$, thus an admissible hyperplane.
\begin{definition} \label{defi.atta}
Let  $H$ be a  $\CA^+$-admissible
hyperplane. A maximal nested set  $M$  such that $\ll M\setminus \CA^+\rr =H$  is
said attached to $H$.
\end{definition}

Given $M$ a $MNS$ for $\CA^+$ attached to $H$, then
$\ll M\setminus \CA^+\rr$ is a $MNS$ for $H\cap \CA^+$.
Therefore maximal nested sets  for an irreducible
set  $\CA^+$ can be determined by induction over the set of $\CA^+$-admissible
hyperplanes.

 For computing the Jeffrey-Kirwan residue, we only need some particular $MNS$'s.
 Let us briefly review the main ingredients.

Fix a total order $\h$ on $\CA^+.$
Let $M=\{S_1,S_2,\ldots,S_k\}$ be a set of subsets of $\CA^+$
and choose in each $S_j$  the element $\alpha_j$ maximal for the
order given by $\h$.
This defines a map $\Theta$ from $M$ to $\CA^+$  and we say that $M$ is \emph{proper} if $\Theta(M)=\overrightarrow{M}$ is a
basis of $V$.
We denote by $\CP(\CA^+)$ the set of   $MPNS.$

So we have associated to every maximal proper nested set $M$ an
ordered basis, by sorting  the set
$\overrightarrow{M}=[\alpha_1,\alpha_2,\ldots,\alpha_r]$
of elements of $\CA^+$.

Let $v$ be an element in $V$ not belonging to any admissible hyperplane.
\begin{definition}
Define $\CP(v,\CA^+)$ to be the set of $M\in  \CP(\CA^+)$ such that
  $v\in \CC(M)=\CC(\alpha_1,\ldots,\alpha_r)$.
\end{definition}
When there is no possibility of confusion we will drop simply write $\CP(v)$ for $\CP(v,\CA^+).$

We are now ready to state the basic formula for our calculations.

\begin{theorem}[ \cite{DCP}] \label{theo.DCP}
Let $\gc$ be a chamber and let $v\in\gc$. Then, for $\phi\in
{\widehat\CR}_{\CA^+}$, we have
$$\JK_\gc(\phi)
  =\sum_{M\in \CP(v,\CA^+)}
    \frac{1}
         {\vol(M)}\Ires_{\overrightarrow{M}}\phi.$$
\end{theorem}

Let us finally sketch the algorithm to determine
 $\CP(v,\CA^+)$ without going to construct all the $MNS$'s.
 If $\CA^+=\CA_1^+\times \CA_2^+$  is reducible, then $\CP(v,\CA^+)$ is the product of the corresponding sets $\CP(v_i,\CA_i^+)$.

 Assume  $\CA^+$ is irreducible.
Let $\theta$  be the highest root of the system $\CA^+$ (for our order $\h$).
We start by constructing all possible $\CA^+$-admissible
hyperplanes $H$ for which $v$ and $\theta$ are strictly on the same side of $H$.
In particular, the   hyperplane $H$  does not contain the highest root.

  Then we  compute the projected vector $proj_Hv$ on $H$ parallel to $\theta$:
 $v=\proj_H v+t \theta$, with $\proj_H v\in H$ and $t>0$  and compute $\CA^+\cap H.$
If $M_H$ is in $\CP(proj_H v,\CA^+\cap H)$, then $M=\{op (M_H),\CA^+\}$ is in
$\CP(v,\CA^+)$. Running through all hyperplanes $H$, for which $v$ and $\theta$ are strictly on the same side of $H$,  we obtain the set $\CP(v,\CA^+)$.
Let us summarize the scheme of the algorithm in
Figure~\ref{algo.MPNS}. Recall that we have as input a regular vector
$v$, and as output the list of all $MPNS$'s belonging to
$\CP(v,\CA^+)$.
\begin{figure}[h!]
\begin{center}
\scriptsize{
\begin{tabular}{|ll|}\hline

for each hyperplane $H$ do&\\
\quad check if $v$ and $\theta$ are on the same side of $H$&\\
\quad \quad if not, then skip this hyperplane&\\
\quad \quad define the projection $\proj_H(v)$ of $v$ on $H$ along $\theta$&\\
\quad \quad write $\CA\cap H$ as the union of its irreducible
       components $I_1\cup\cdots\cup I_k$&\\
\quad \quad write $v$ as $v_1\oplus\cdots\oplus v_k$ according
      to the previous decomposition&\\
\quad \quad \quad for \= each $I_j$ do&\\
\quad \quad\quad  compute all MPNS's for $v_j$ and $I_j$&\\
\quad \quad\quad  collect all these MPNS's for $v_j$ and $I_j$&\\
\quad \quad \quad end of loop running across $I_j$'s&\\
\quad  collect all MPNS's for the hyperplane $H$ by taking the product of $\CP(I_j,v_j)$&\\
 end of loop running across $H$'s&\\
return the set of all MPNS's for all hyperplanes &\\
\hline
\end{tabular}}
\end{center}
\caption{Algorithm for MPNS's computation (general case)}
\label{algo.MPNS}
\end{figure}





In our program, we run this algorithm for an element $v$ not in any admissible hyperplane,  without knowing in advance if $v$ belongs to the cone $\CC(\CA^+)$. The algorithm returns a non empty set if and only  if $v$ belongs to
$\CC(\CA^+)$.

 \subsection{The Kostant function: another formula for subsystems of $A_r^+$}

In this article, we will be using partition functions for lists  $\Delta^+(A,B)$ described in the Example \ref{Ar}.
These lists are sublists of a system of type $A_r^+$ (with $r=p+q-1$).  In residue calculation, we can use change of variables and thus use a formula for which iterated residues will be easier to compute.

Let us describe this formula. We will describe it for sublists, eventually, with multiplicities of a system $A_r^+$.
 We take the notations of Example \ref{Ar}.

Let $\Phi$ be a sequence of  vectors  generating $V$ and of the form $(e_i-e_j), 1\leq i<j\leq (r+1)$, eventually with multiplicities.
 Let $m_{i,j} $ ($i<j$) be the multiplicity of the vector $e_i-e_j$
in $\Phi$   and  define
$t_j=m_{j,j+1}+\cdots+m_{j, r+1}-1.$
We recall our identification of $V$ with $\R^r$ and of
$U=V^*$ with $\R^r$ defined by duality. In this way, as we already observed  the root $e_i-e_j$ ($1\leq i<j\leq r$) produces the linear
function $u_i-u_j$ on $U$, while the root $e_i-e_{r+1}$ produces
the linear function $u_i$.

We are now ready to give another formula for the Kostant function in this situation.

\begin{theorem} \label{theo.flow}Let $\gc$ be a chamber of  $\CC(\Phi)$.
 Let $h=\sum_{i=1}^{r+1}h_i e_i=[h_1,\ldots,h_r]$,  then

$$ {\bf N}_{\Phi}^{\gc}(h)=\vol \left( V/V_\Z, dh \right)\JK_{\gc}( f_{\Phi}(h)(u)), {\rm where}$$ \ \

$$f_{\Phi}(h)(u)=\frac{\prod\limits_{i=1}^r(1+u_i)^{h_i+t_i}}{\prod\limits_{i=1}^ru_i^{m_{i,r+1}
} \prod\limits_{1\leq i< j\leq r}(u_i-u_j)^{m_{i,j}}}$$

\end{theorem}

Thus, when $h\in V_\Z\cap (\gc-Z(\Phi))$, we  have
$$ N_{\Phi}(h)= {\bf N}_{\Phi}^{\gc}(h).$$

\begin{example}
\begin{itemize}
\item If $\Phi=A_r^+$, then
$$ f_{A^+_r}(h)(u)=\frac{\prod_{i=1}^{r}(1+u_i)^{h_i+r-i}}
     {\prod_{1\leq i<j\leq r}(u_i-u_j)
      \times
      \prod_{i=1}^{r}u_i}.$$
\item
 Let $ p, q$ integers such that $p + q = r + 1$ and $\Phi$ be the system of
positive noncompact root for $A^+_r$   defined by $\Phi=\{e_i - e_j , 1 \leq i \leq p, \ p + 1 \leq
j \leq  r + 1\}.$
That is $\Phi=\Delta^+(A,B)$ with $A=[1,\ldots,p]$ and $B=[p+1,\ldots, p+q]$.

Then
$$f_{\Phi}(h_1,h_2,\ldots,
h_r)(u)=$$$$\frac{(1+u_1)^{h_1+q-1}\cdots(1+u_p)^{h_p+q-1}(1+u_{p+1})^{h_{p+1}-1}\cdots (1+u_{p+q-1})^{h_{p+q-1}-1}}{(u_1-u_{p+1})\cdots (u_1-u_{p+q-1})\cdots(u_p-u_{p+1})\cdots(u_p-u_{p+q-1})u_1u_2\cdots u_p}$$
\end{itemize}
\end{example}
\begin{proof}
The function $K(0,h)(u)=e^{\ll h,u\rr}/
  \prod_{\alpha\in\Phi}(1-e^{-\ll\alpha,u\rr})$
computed for the system $\Phi$
is
$$K(0,h)(u)= \frac{e^{h_1  u_1}e^{h_2 u_2}\cdots e^{h_r u_r}}
{\prod\limits _{i=1}^r(1-e^{-u_i})^{m_{i,r+1}}\prod\limits _{1\leq i<
j\leq r}(1-e^{-(u_i-u_j)})^{m_{i,j}}}$$

Note that the change of variable $1+z_i=e^{u_i}$ preserves the
hyperplanes $u_i=0$ and $u_i=u_j$ and that  $z_i=e^{u_i}-1$ leads to $dz_i=e^{u_i}du_i=(1+z_i)du_i.$   Thus after the change of variable  we get
the required formulae.
 \end{proof}

\subsection{Computation of Kostant partition function: general scheme}\label{Kostantgeneral}

\subsubsection{Numeric}\label{numeric}

We have as input $\CA^+$ a sequence of vectors  in our lattice $V_\Z$, a vector $h\in V_\Z$, and we want to compute  $N_{\CA^+}(h)$.  We will compute it  by  !!
 $$N_{\CA^+}(h)= N_{\CA^+}(h+\rho_n -\rho_n).$$
We mean: Let $\tau$ be any tope such that $h'=h+\rho_n$  belongs to the closure of  $\tau$. Using Theorem \ref{theoremtope}
 then
$$N_{\CA^+}(h)=N_{\CA^+}(h'-\rho_n)={\bf N}_{\CA^+}^{\tau}(h'-\rho_n).$$

To compute a tope $\tau$ containing $h'$, we can move $h+\rho_n$ in any generic direction $\epsilon$.\label{deform}

Here is an outline of the steps needed to compute the number
$N_{\CA^+}(h)$  by the formula
$N_{\CA^+}(h)={\bf N}_{\CA^+}^{\tau}(h).$

{\bf Input}:  a vector $h\in V_\Z$, and $\CA^+$ a sequence of vectors in $V_\Z$.

{\bf Output}:  the number   $N_{\CA^+}(h)$.

\begin{itemize}
\item {\bf Step 1} Compute the  Kostant function $$K(h)=K(0,h)=\frac{e^{h}}{\prod_{\alpha\in \CA^+}(1-e^{-\alpha})}$$
or more generally compute a set $F$ and   the functions  $K(G,h)$ for $G\in F$.
\item  {\bf Step 2} Find a small vector $\epsilon$ so that if $h$ is in $V_\Z$, the vector $h+\rho_n+\epsilon$ does not belong to any admissible hyperplane.
Thus the vector $h_{reg}=h+\rho_n+\epsilon$ is in a unique tope $\tau$.
 The procedure to obtain $h_{reg}$ is called $DefVector_{nc}(h).$
 \item  {\bf Step 3}
Compute  the set $All:=\CP(h_{reg},\CA^+)$  as explained in Fig.\ref{algo.MPNS}.
\item  {\bf Step 4}
    Compute   ${\bf N}_{\CA^+}^{\tau}(h)$
       by computing  the iterated residues  of $K(G,h)$ associated  to the various ordered basis $\overrightarrow{M}$ for $M$ varying in the set $All$.
       That is  compute    the number

\fbox{\parbox{8cm}{\[out:=\sum_{G\in F} e^{\ll h,2\pi\sqrt{-1}G\rr}\sum_{M\in All} \Ires_{\overrightarrow{M}} K(G,h)\]}}

\noindent where $\overrightarrow{M}$ is the ordered basis attached to $M$.
\medskip
\medskip

  \noindent Then  \fbox{$N_{\CA^+}(h)=out$}

\end{itemize}

\subsubsection{Symbolic}\label{alsymb}
The previous calculation runs with symbolic parameters.
If $hfix$ is an element in $V_\Z$, we might want to find a tope $\tau$ such that $hfix$ belongs to the closure of $\tau$.
Then
$$N_{\CA^+}(h)={\bf N}_{\CA^+}^{\tau}(h)$$
will be valid whenever $h$ is in the closure of $\tau$. Here is the outline of the algorithm.

{\bf Input:}  $hfix$ is an element in $V_\Z$ and $\CA^+$ a sequence of vectors.

 {\bf Output:} A domain $D\subset V$ and an exponential polynomial function  $P(h)$ on $V$.

 The domain $D$ is a closed convex cone in $V$ (described by linear inequations) such that $hfix$ is in $D$. The formula $P(h)=N_{\CA^+}(h)$ is valid whenever $h\in D\cap V_\Z$.

\begin{itemize}
\item  {\bf Step 1}  Consider $h$ as a parameter and
compute the  Kostant function $K(h)(u)$
as a function of $(h,u)$
given by $$K(h)(u)=K(0,h)(u)=\frac{e^{<h,u>}}{\prod_{\alpha\in \CA^+}(1-e^{-<\alpha,u>})}$$
or more generally compute a set $F$ and   the functions   $K(G,h)(u)$ for $G\in F,$ as function of $(h,u)$.
\item  {\bf Step 2}  Find a small vector $\epsilon$ so that if $hfix$ is in $V_\Z$, then the  vector $hfix_{reg}:=hfix+\rho_n+\epsilon$ does not belong to any admissible hyperplane.

Compute the domain $D:={\overline \tau}$ where $\tau$ is the unique tope $\tau$ containing $hfix_{reg}$.
\item  {\bf Step 3} Compute  the set $All:=\CP(hfix_{reg},\CA^+)$  as explained in Fig.\ref{algo.MPNS}.
\item  {\bf Step 4}  Compute   ${\bf N}_{\CA^+}^{\tau}(h)$
       by computing  the iterated residues  of $K(G,h)$ associated  to the various ordered basis $\overrightarrow{M}$ for $M$ varying in the set $All$, here $h$ is treated now as a parameter.

\noindent That is we compute
$$out:=\sum_{G\in F} e^{\ll
h,2\pi\sqrt{-1}G\rr}\sum_{M\in All} \Ires_{\overrightarrow{M}} K(G,h)$$
where $\overrightarrow{M}$ is the ordered basis attached to $M$.
The output $out$ is an exponential polynomial function $P(h)$ of $h$ and once again we compute

\fbox{\parbox{7cm}{ $N_{\CA^+}(h)=P(h)=out, \ \forall h\in D\cap V_{\Z}$}}

\noindent The domain $D$ is a rational polyhedral cone which includes $hfix$.
\end{itemize}
In practice, this works only for small dimensions and when $\CA^+$ is not too big.

\noindent We will program variations of these algorithms, with less ambitious goals.

 \subsection{Computation of Blattner formula: general scheme.}\label{general}

In this subsection,  we  summarize the steps to compute  Blattner's  formula and the general scheme  to obtain the region of polynomiality. The relative algorithms  will be outlined in Section \ref{programs}.

Let $G,K,T$ be given as in Section \ref{BF}.
Let  $\Delta_n\subset \mathfrak t^*$  be the list of noncompact roots.

Our inputs are  $\lambda \in P_{\mathfrak g}^r$
and $\mu \in P_{\mathfrak k}^r$.
The goal is the study of the  function $\mu\to
 m^{\lambda}_{\mu}.$
Let $\CA^+=\Delta_n^+(\lambda)$
and recall  that in this case  a  $\CA^+$-admissible hyperplane is called  a noncompact wall.

We  use {\bf Blattner's formula}. In our notations:
 \begin{equation}   m^{\lambda}_{\mu} =\sum_{w \in \mathcal W_{c}} \epsilon(w)  N_{\Delta_n^+(\lambda)}(w\mu-\lambda-\rho_n)
\end{equation}

\subsubsection{Numeric}\label{Bnum}

{\bf Input}    $\lambda \in P_{\mathfrak g}^r$
and $\mu \in P_{\mathfrak k}^r$.

\noindent {\bf Output}  a number.

The algorithm is   clear:

\begin{itemize}
\item  Compute $\CA^+=\Delta_n^+(\lambda)$ and $\rho_n$.

  \item Compute the Kostant function $K(G,h)$ for this system $\CA^+$.
  \item Compute a finite set $F$ satisfying condition $(\bf C) $.
  \item Compute a  small element $\epsilon$ such that  $\mu_{reg}=\mu+\epsilon$  does not belong to any affine hyperplane of the form $w\lambda+H$ where $H$ is a noncompact wall, $w\in \CW_c$.

\item Compute for all $w\in \CW_c$
the  number
$$contribution_w:=N_{\CA^+}(w\mu-\lambda-\rho_n)$$ using the algorithm described in \ref{numeric}.

That is compute

$All_w:=\CP(w*(\mu_{reg})-\lambda,\CA^+)$  as explained in Fig.\ref{algo.MPNS}.

Then compute
$$contribution_w:= \sum_{G\in F} e^{\ll
h,2\pi\sqrt{-1}G\rr}\sum_{M\in Allw} \Ires_{\overrightarrow{M}} K(G,w\mu-\lambda-\rho_n).$$
ENDs

\item Finally compute

\fbox{\parbox{8cm}{$$out:=m^{\lambda}_{\mu} =\sum_{w \in \mathcal W_{c}} \epsilon(w)* contribution_w$$}}

\end{itemize}

 Let us comment briefly:
 If $w (\mu_{reg})-\lambda$ is not in the cone generated by non compact positive roots, the set $All_w$ is an empty set.
 In particular we may restrict the computation by diverse consideration
to {\bf valid permutations}  which have some chance to give a non empty set, see Section \ref{valid}.)

\subsubsection{  Symbolic}\label{symbolic}
The preceding calculation runs with symbolic parameter and we take advantage of this to find regions of polynomiality.

Let's explain how.

 Let $\a$ be a chamber in $\mathfrak t^*$ for the system $\Delta$ of roots of $\mathfrak \g$.
  Let $U$  be the open set of $(\lambda,\mu)\in \a\times \mathfrak t^*$ such that
 $w\lambda-\mu$ does not belong to any non compact wall.
 Let  $(\lambda_0,\mu_0)\in U$.
 Then we define $R(\lambda_0,\mu_0)$ to be the closure of connected component of $U$ containing $(\lambda_0,\mu_0)$.
This region is a cone in $\mathfrak t^*\times \mathfrak t^*$ with non empty interior and can be described by linear inequalities in $\lambda,\mu$.

For this domain we can compute a polynomial formula and
state the following result.

If $\lambda$ varies in $\a$, the systems $\Delta_n^+(\lambda),\Delta_c^+(\lambda)$ determined by $\lambda$ remains the same. We denote it by $\Delta_n^+$, $\Delta_c^+$.

Furthermore, if $(\lambda,\mu)\in R(\lambda_0,\mu_0)$, for any $w\in \mathcal W_c$, the element $w\lambda-\mu$ lies in a  tope $\tau_w$ for the system $\Delta^+_n$ which depends only of $w$.

\begin{theorem}
The domain $R(\lambda_0,\tau_0)$ is a domain of polynomiality for the Duistermaat-Heckman measure and thus is a domain of quasi-polynomiality  for the multiplicity function
$m^{\lambda}_{\mu}$.

More precisely, for
 $(\lambda,\mu)\in R(\lambda_0,\mu_0)\cap (P_\mathfrak
g^{r}\times P_\mathfrak \k^r),$
we have

\begin{equation}\label{symb}
m^{\lambda}_\mu= \sum_{w\in {\mathcal W}_c}\epsilon(w)
{\bf N}^{\tau_w}_{\Delta_n^+}(w\lambda-\mu-\rho_n).
\end{equation}

\end{theorem}

The right hand side  of Equation \ref{symb} is a quasi polynomial
function of $\lambda,\mu$, antisymmetric in $\mu$.
It takes positive values if $\mu$ is dominant for $\Delta_c^+$.
Recall that, in this case, the multiplicity of $\mu$ on $\lambda$ is the absolute value of the function $m_\mu^{\lambda}$ above, (Sec.\ref{BF}).

Of course, the symbolic calculation above, with the present approach,  is limited  to very small examples.

Also,  we are not able to determine the largest domains where  the function $m_\mu^\lambda$ is given by a quasi polynomial formula.

%

\subsubsection{Asymptotic directions}\label{asymp}

We address now a simpler problem. We have the same setting  that in the previous section, but we are now testing only the noncompact walls  crossing in one fixed direction $\vec{v}$.

Let $\mu_0,\lambda_0$ be given, with $\lambda_0\in\a\cap P_\g^r$ an Harish-Chandra parameter, and $\mu_0\in \a_c\cap P_\k^r$.
Let $\vec{v}\in \a_c$ be integral. We will do the calculation of
$m^{\lambda_0}_{\mu_t}$ when $\mu_t =\mu_0+t \vec{v}$, with $t\geq 0$, is in the half-line in the direction $\vec{v}$. In the application $\mu_0$ will be the lowest $K$-type $\lambda_0-\rho_n$ of our discrete series $\pi^{\lambda_0}$.
We compute the values $t_i$ where $\mu_0+ t \vec{v} -w\lambda_0$  cross   a non compact   wall  (other than the ones
 which may contain the line $\mu_0+ t \vec{v} -w\lambda_0$). These are the values where the line $\mu_t$ may cross the domains of quasipolynomiality described above.
We order this finite set of  values $0\leq  t_1 <t_2 <t_i <\cdots <t_s.$
Consider the interval $I_i=[t_i,t_{i+1}], \ \ 0\leq i\leq s,$ where $t_0=0$ and $t_{s+1}=\infty.$

Consider  $I_i\cap \Z$,  an "interval" in $\Z$, described by two integers $[a_i,b_i]$, with  $a_i=ceil(t_i)$ and $b_i=floor(t_{i+1})$.

The "interval"  $I_i\cap \Z$ can also be reduced to a point.

Then we   find  exponential polynomial function $P^i(t)$ on $\R$  such that
$m^{\lambda}_{\mu_t}$ is equal to $P^i(t)$ for $t \in I_i\cap \Z$.

If particular, $\vec{v}$ is an asymptotic direction, if and only if the last quasipolynomial  $P^{s}(t)$
 does not vanish.

The algorithm is as follows.

For each consecutive value $t_i,t_{i+1}$, choose $\mu_r=\mu+t_r v$ with $t_i<t_r<t_{i+1}$. Then,  move very slightly $\mu_r$ in $\mu_r^{\epsilon}$. Then for each $w\in \mathcal W_c$, $w\mu_r^{\epsilon}-\lambda_0$ lies in a tope $\tau_i^w$ for $\Delta_n^+(\lambda_0)$.
Then
$$P_{i}(t)=\sum_{w} \epsilon(w) {\bf N}_{\Delta_n^+(\lambda)}^{\tau_i^w}(w\mu_t-\lambda-\rho_n).$$

The right hand side of this formula is an exponential polynomial function of $t\in \Z$.

The algorithm implementing this procedure is described in  Fig.\ref{algo.pol}.
Let us  remark that  our algorithm implementation is for type $A_r$ and thus the $P_i$ are polynomials.

\section{ Blattner's formula for $U(p,q)$}\label{Upq}

\subsection{Non compact positive  roots}\label{noncp}

With the notation of  Section \ref{blattner}  we let
$G=U(p,q)$ and  $K=U_p\times U_q$  be a maximal compact subgroup.

Let $E$ be a $p+q$-dimensional vector space with basis $e_i$  ($i=1$, \ldots, $p+q$) and $V$ as in Ex. \ref{Ar}).
Let $r=p+q-1$. Consider the set  of roots
 $$\Delta=\pm\{e_i-e_j\,|\,1\leq i<j\leq {p+q}\}.$$

We then choose $T$ to be the diagonal subgroup of $U(p,q)$, and identify $\t^*$
with $E$. In this identification the lattice of weights is identified with $\Z^{p+q}$:
 the element $(n_1,\ldots, n_{p+q})$ giving rise to the character
$$t=(\exp(i\theta_1), \ldots, \exp(i \theta_{p+q}))\to e^{i n_1\theta_1}\cdots e^{i n_{p+q}\theta_{p+q}}.$$

The system  of compact roots $\Delta_c$  is

$$\Delta_c=\pm\{e_i-e_j\,|\,1\leq i<j\leq {p} \}\cup \pm\{e_i-e_j\,|\,p+1\leq i<j\leq {p+q} \}.$$

                The system of non compact roots   is

$$\Delta_n=\pm\{e_i-e_j\,|\,1\leq i \leq {p}, p+1\leq j \leq {p+q} \}.$$

Let $\lambda$ be the Harish Chandra parameter  of a discrete series for $G$  and $\mu$ the Harish-Chandra parameter  of a finite dimensional irreducible representation of $K$.

 Because discrete series are equivalent under the action of the Weyl group of $K$,
  then we may assume that $\lambda=\ [\alpha,\beta]\ $ where $\alpha= \sum_{i=1}^{p}\alpha_i e_i=[\alpha_1,\ldots,\alpha_p], \alpha_1>\alpha_2>\cdots>\alpha_p$ and $\beta=\sum_{i=p+1}^{p+q}\beta_i e_i=[\beta_1,\ldots,\beta_q],\beta_1>\beta_2\cdots>\beta_q$.

Here $\alpha_i,\beta_j$ are integers if $p+q$ is odd, or half-integers if $p+q$ is even,
that is we fix as system of positive compact roots the system $\Delta_c^+=\{e_i-e_j\,|\,1\leq i<j\leq {p} \}\cup\{e_i-e_j\,|\,p+1\leq i<j\leq {p+q} \}.$

We parametrize $\mu\in P_\mathfrak k^r$ by another couple

$$\mu:=[a,b]
=[[a_1, a_2,\ldots,a_p],[b_1,\ldots,b_q]]$$
with $a_1>\cdots > a_p$
and $b_1>\cdots> b_q$.

Here $a_i$ are integers if $p$  is odd, half-integers if $p$ is even.
Similarly $b_j$ are integers if $q$ is odd, half-integers if $q$ is even.

As the center of $G$ acts by a scalar in an irreducible representation,
 we need that the sum of the coefficients of $\lambda$ has to be equal to the sum of the coefficients of $\mu$
for the multiplicity of $\mu$ in $\pi^{\lambda}$ to be non zero.  Thus $\lambda-\mu$ is in $V$, (see Ex. \ref{Ar}).

We now parametrize the different dominant chambers of $\mathfrak t^*$ modulo the Weyl group of $K$ by a subset $A$ of $[1,2,\ldots, r+1]$ of cardinal $p$. Let  $B$ its complementary subset in $[1,2,\ldots,r+1]$.

To visualize  $A,B$ we write a  sequence of lenght $p+q$ of elements $a,b$ with $a$ in the places of $A$, $b$ in the places of $B$: for example if $A=[3,5]$ and $B=[1,2,4]$, then we write $[b,b,a,b,a]$ or simply $bbaba$.
Now we use this visual aid and describe a permutation $w_A$ of the index $[1,2,\ldots, p+q]$, by putting the index $[1,\ldots,p]$ in order and in the places marked by $a$, and the remaining indices $[p+1,\ldots,p+q]$ in order and in the places marked by $b$, precisely $w_A: [1,2,3,4,5] \rightarrow [3,4,1,5,2]$.
The elements $w_A$ where $A$ varies describe a system of representatives of  $\Sigma_{p+q}/(\Sigma_p\times \Sigma_q)$, ($\Sigma_n$ being the permutations on $n$ letters),  that is also the chambers of $\mathfrak t^*$ for $\Delta(\g,\t)$ modulo $\mathcal W_c.$ In the above the chamber is described by $\{h=[h_1,h_2,h_3,h_4,h_5] \ \ \ \ h_3>h_4>h_1>h_5>h_2\}$

Let $\a_{standard}$ be the chamber $\alpha_1>\alpha_2>\cdots>\alpha_p>\beta_1>\beta_2\cdots>\beta_q$.

Then  if $\lambda\in w_A\a_{standard}$, we have $\Delta_c^+(\lambda)=\Delta_c^+$ and $\Delta_n^+(\lambda)$ is isomorphic to $\Delta^+(A,B)$, by relabeling the roots via $w_A^{-1}.$ The next example will clarify the situation.
The subset  $A$ can be read  from $\lambda$: we reorder completely the sequence $\lambda$ and define $A$ as the indices  where the first $p$ elements of $\lambda$  are relocated.

\begin{example}
Let $G=U(2,3)$ with compact roots $\Delta_c=\pm\{e_1-e_2, e_3-e_4,e_3-e_5,e_4-e_5\}$ and  noncompact roots $\Delta_n=\pm\{ e_1-e_3,  e_1-e_4,e_1-e_5,e_2-e_3 ,e_2-e_4,e_2-e_5\}.$
Let $\lambda=[\alpha,\beta]$ with $\alpha=[4,2]=4e_1+2e_2$ and
$\beta=[6,5,3]=6e_3+5e_4+3e_5$.

Then  $A=[3,5]$, $B=[1,2,4]$ that is the configuration $bbaba$.
The system of non compact positive roots  for $\lambda$
is $e_3-e_1, e_4-e-1,e-1-e_5,e_3-e_2,e_4-e_2,e_5-e_2$, isomorphic to $\Delta^+(A,B)$ by  the relabeling of the roots suggested by $w_A^{-1}$, that is $e_3=f_1, e_4=f_2,e_1=f_3,e_5=f_4,e_2=f_5.$
 \end{example}
Thus relabeling the roots, our calculations will be done for $\Delta^+(A,B)$ inside $A_r^+$,
where $\Delta^+(A,B)$ is given in Example \ref{Ar}.

Remark here that $\Delta^+(A,B)$ is irreducible if  $p$ and $q$ are strictly greater than $1$.
In contrast, when $p$ or $q=1$, the system is fully reducible.
Consider for example  the case  $p=1$.
 \begin{example}\label{ex}
\end{example}

 In this case $A$ has only $1$ element and  the system $\Delta^+(A,B)$ has $r$ elements and is  a base of $V$.
Thus $\Delta^+(A,B)$ is fully reducible in the direct sum of $p+q-1$ one dimensional systems.

For example take $U(1,r)$ with $A=[1]$ and $B=[2,\ldots,r+1]$. Then $\Delta^+(A,B)=\{e_1-e_2,e_1-e_3,\ldots,e_1-e_{r+1}\}$
is isomorphic to $A_1^+\times A_1^+\times \cdots \times A_1^+$.

\bigskip

Remark that  when $A_1,A_2$ have the same number of elements, although  the system of noncompact roots $\pm \Delta_n^+(A_1, B_1)$ and $\pm \Delta_n^+(A_2,B_2)$ are clearly isomorphic,   the combinatorial properties of $\Delta_n^+(A,B)$ may vary.

\noindent For example,  (see Figure \ref{figurea3nonc}), if $A=[1,2]$,  $B=[3,4]$, the cone generated by the non compact roots has basis a square and is not a simplicial cone. If $A=[1,3]$ and $B=[2,4]$, then the cone generated by the non
compact roots is the simplicial cone generated by
$e_1-e_2,e_2-e_3,e_3-e_4. $

\subsection{Algorithm to compute  $MPNS$:  the case of $\Delta^+(A,B)$}

With the notations  of Ex.\ref{Ar}, we denote by  $A$  a proper subset of $[1,2,\ldots, r+1]$ (with $r=p+q-1$) and by
$B$  the complementary subset to $A$ in $[1,2,\ldots r+1]$.

Given $v\in V$, and not on any admissible hyperplane,  we describe the algorithm to compute $\mathcal P(v,\Delta^+(A,B))$.

If $p$ or $q=1$, roots $\alpha$ in
$\Delta^+(A,B)$ form a basis on $V$, thus there is only one maximal nested set $M=\{\{\alpha\}, \alpha\in  \Delta^+(A,B)\}$.
Thus  $\mathcal P(v,\Delta^+(A,B))$ is empty or equal to $\{\vec M\}$ depending if $v$ belongs to the cone generated by $\Delta^+(A,B)$, or not. This is very easy to check.

If $p>1$ and $q>1$ , we determine the set   $\mathcal P(v,\Delta^+(A,B))$  by induction, going to admissible hyperplanes.

If $L\subsetneq  [1,2,\ldots,r+1]$  is a proper  subset of $[1,2,\ldots, r+1] $, we will also use the notation $L'=[i\notin L\,|\,1\leq i\leq r+1]$  for the complement of $L$. We  denote by  $H_ L:=\{v\in V\,|\  \sum_{i\in L} v_i=0\}$ the hyperplane determined by $L$;  the hyperplane $H_L$ is equal to
the hyperplane $H_{L'}$ determined by $L'$.

 \noindent It is  very simple to describe $\Delta^+(A,B)$-admissible hyperplanes, that is noncompact walls.
The description is  an adaptation of the  $A^+_r$-admissible hyperplanes  that appear in \cite{BBCV}.

Keeping $A$ fixed, with $|A|\neq 1,r$, we  consider hyperplanes $H_L$ indexed by subsets $L\subset 1,2,\ldots,r+1]$ with the following properties:
\begin{itemize}\item if  $|L|\neq 1$ or $r$, then $H_L$ is a   noncompact wall if and only if  both $A$ and $B$ intersect $L$ and $L'$.
In this case $\Delta_n^+(A,B)\cap H_L$ is  the product of two systems $\Delta_n^+(A\cap L, B\cap L)\times \Delta_n^+(A\cap L', B\cap L')$ and thus reducible.
\item if $L$ is of cardinal $1$, then $H_L$ is a   noncompact wall. In this case $\Delta_n^+(A,B)\cap H_L$ is   $\Delta_n^+(A\cap L',B\cap L')$ and thus irreducible.
\end{itemize}

At this point to compute the $MNPS$ or better, as we explained the $\overrightarrow{M}'s$,  we can proceed as in Fig.\ref{algo.MPNS}.
The algorithm is outlined in Fig.\ref{algo.nonc}.

We conclude with the following observation.
A necessary and sufficient condition for the set $MPNS(v,\Delta_n^+(A,B))$ to be non empty is that
$v$ belongs to the cone generated by $\Delta_n^+(A,B)$. As far as we know, the equations of this cone are not known, except in a few cases.
It is clearly necessary that $v$ belongs to the simplicial cone generated by all positive roots. To speed up the calculations, we check this condition at each step of the algorithm.

We conclude with a simple example with $p=2,q=2$ and
$A=[1,2]$, $B=[3,4]$. We follow the outline described in Fig.\ref{algo.MPNS}. The highest non compact root is $\theta:=e_1-e_4$.
There are $3$ noncompact walls  not containing the highest root.
$L=[1], [4],[1,3]$
We choose a vector   $v=[4,3,-2,-5]$ not on any noncompact walls.
 Then $[1],[4],[1,3]$ are all such that $v$ and $\theta$ are on the same side.

For $L=[1]$, the $v$ projection do not belong to the cone generated by $\Delta_n^+(A,B)\cap H_L={[e_2-e_3,e_3-e_4]}.$

For $L=[4]$, we obtain the element $M:=\{[1,2,3,4],[1,3],[2,3]\}$ in $\CP(v,\Delta^+_n(A,B)).$

For $L:=[1,3]$, we obtain the element $M=\{[1,2,3,4],[1,3],[2,3]\}$ in $\CP(v,\Delta^+_n(A,B)).$

\subsection{Valid permutations}\label{valid}

 Let $w\in \mathcal W_c$.
  Remark that if $w\mu-\lambda$  does not belong to the cone of non compact positive roots, then  the term corresponding to $w$ in Blattner formula is equal to $0.$
It is important to minimize the number of terms in Blattner formula.
To this purpose, we use a weaker condition:
     we say that  $w \in \mathcal W_{c}$ is a valid element if $w\mu-\lambda$ is in the cone spanned by  (all)  positive roots.
Thus if $w$ is not valid, the corresponding term to $w$ in Blattner formula is equal to $0$.
As there is a simple description of the faces of cone spanned by all positive roots (it is the simplicial cone dual to the simplicial cone generated by fundamental weights),
there is a simple algorithm that constructs valid permutations one  at the time depending on the conditions they have to satisfy, instead of listing all the elements of $\mathcal W_c$.
The corresponding algorithm is used in \cite{C1},\cite{C2}, and we just reproduced it.

\section{Examples}
\begin{example}\end{example}
 We consider the discrete series representation indexed by $\lambda$ and we test for the multiplicity $m_{\mu}^{\lambda}$ where $\mu$ is a K type.
 We write $\mu_{lowest}$ for the lowest $K$-type.  We use the algorithm whose command is :
\medskip

\noindent {\bf  $>$discretemult($\lambda$,$\mu$,p,q)}

\noindent{\scriptsize
\begin{tabular}{|| l  |  l |  l  | l  | |} \hline
\multicolumn{4}{|| c ||} { $m_{\mu}^{\lambda}$:  \bf  numeric case  }\\ \hline
{\bf Group}& {\bf Input}& {\bf Output}& {\bf Time}\\  \hline

&$\lambda$=[[31/2,15/2,9/2],[5/2,3/2,-5/2]]&&\\  \cline{2-4}
& $\mu_{lowest}$=[[17, 9, 6], [1, 0, -4]]&{\scriptsize {\bf  $1$}}&{\scriptsize {\bf $0.026$} sec. }\\  \cline{2-4}
\raisebox{1.5 ex}[0pt]  {\bf U(3,3)}& $\mu$=[[1017, 1009, 1006], [-999, -1000, -1004]]&{\scriptsize $9$}&{\scriptsize {\bf $0.97$} sec.}\\ \cline{2-4}
& $\mu$=[[100017, 10009, 10006], [-9999, -10000, -100004]]&{\scriptsize $9$}&{\scriptsize {\bf $0.91$} sec.}\\  \hline
&$\lambda$=[[31/2, 19/2, 11/2], [ 15/2, 7/2, -37/2]]&&\\  \cline{2-4}
& $\mu_{lowest}$=[[17, 11, 6], [7, 2, -20]]&{\scriptsize {\bf  $1$}}&{\scriptsize {\bf $0.073$} sec. }\\  \cline{2-4}
\raisebox{1.5 ex}[0pt]  {\bf U(3,3)}& $\mu$=[[1017, 1011, 1006], [-993, -998, -1020]]&{\scriptsize $275$}&{\scriptsize {\bf $0.529$} sec.}\\ \cline{2-4}
& $\mu$=[[100017, 10011, 10006], [-9993, -9998, -100020]]&{\scriptsize $11700255$}&{\scriptsize {\bf $0.538$} sec.}\\  \hline
&$\lambda$=[[11/2,7/2,3/2,-1/2],[9/2,5/2,1/2,-3/2]]&&  \\  \cline{2-4}
{\bf U(4,4)}& $\mu_{lowest}$=[[15/2, 9/2, 3/2, -3/2], [11/2, 5/2, -1/2, -7/2]]&{\scriptsize {\bf  $1$}}&{\scriptsize  {\bf $0.565$} sec.}\\  \cline{2-4}
& $\mu$=[[2015/2, 9/2, 3/2, -3/2], [11/2, 5/2, -1/2, -2007/2]]& {\scriptsize{\bf  $120495492015$}}&{\scriptsize {\bf $3.493$} sec.} \\  \hline
&$\lambda$=[[11/2, 9/2, 7/2, 5/2], [3/2, 1/2, -1/2, -3/2]]&&  \\  \cline{2-4}
{\bf U(4,4)}& $\mu_{lowest}$=[[15/2, 13/2, 11/2, 9/2], [-1/2, -3/2, -5/2, -7/2]]&{\scriptsize {\bf  $1$}}&{\scriptsize  {\bf $0.334$} sec.}\\  \cline{2-4}
& $\mu$=[[20015/2, 2013/2, 211/2, 29/2], [-21/2, -203/2, -2005/2, -20007/2]]& {\scriptsize{\bf  $1$}}&{\scriptsize {\bf $273.719$} sec.} \\  \hline
&$\lambda$=[[5, 3, 1, -1, -3], [4, 2, 0, -2]]&&  \\  \cline{2-4}
{\bf U(5,4)}& $\mu_{lowest}$=[[7, 4, 1, -2, -5], [11/2, 5/2, -1/2, -7/2]]&{\scriptsize {\bf  $1$}}&{\scriptsize  {\bf $3.952$} sec. }\\  \cline{2-4}
& $\mu$=[[1007, 4, 1, -2, -5], [11/2, 5/2, -1/2, -2007/2]] &{\scriptsize{\bf  $120495492015$}}&{\scriptsize {\bf $13.752$} sec.}\\  \hline
&$\lambda$=[[11/2,7/2,3/2,-1/2,-5/2],[9/2,5/2,1/2,-3/2,-7/2]]&&  \\  \cline{2-4}
{\bf U(5,5)}& $\mu_{lowest}$=[[8, 5, 2, -1, -4], [6, 3, 0, -3, -6]]&{\scriptsize {\bf  $1$}}&{\scriptsize  {\bf $51.910$} sec. }\\  \cline{2-4}
& $\mu$=[[106,4,2,0,-102],[104,2,0,-2,-104]]& {\scriptsize{\bf  $1458704380546472381$}}&{\scriptsize {\bf $163.104$} sec.}\\  \hline
\end{tabular}}
\begin{example}\end{example}
 We consider the discrete series representation indexed by $\lambda,$ a direction $\vec{v}$ and we test for the multiplicity $m_{\mu+t\vec{v}}^{\lambda}$ where $\mu=\mu_{lowest}$ is the lowest $K$-type. We use the algorithm whose command is

\noindent{\bf  $>$ function$_{-}$discrete$_{-}$mu$_{-}$direction$_{-}$lowest$_{-}$($\lambda$,$\vec{v}$,p,q)
}

\noindent For completeness we list $\mu_{lowest}$  relative to each example.
\medskip




\noindent{\scriptsize
\begin{tabular}{|| c  |  l  |   l  |    l  | l  ||} \hline
\multicolumn{5}{|| c ||}{ $m_{\mu+t\vec{v}}^{\lambda}\ ,\  t \in \N$:  \bf   asymptotic  case}\\ \hline
{\bf Group}& {\bf Input}&$m_{\mu+t\vec{v}}^{\lambda}$& {\bf Output }& {\bf Time}\\  \hline
& $\lambda$=[[9, 7], [-1, -2, -13]]&&&\\  \cline{2-2}
& $\mu_{lowest}$=[[21/2, 17/2], [-2, -3, -14]]&&&\\  \cline{2-5}
& $\vec{v}$=[[1, 0], [-1, 0, 0]]&{\scriptsize $m_{\mu+t\vec {v}} ^\lambda= \left\{
\begin{array}{r @{\quad {\text {if}} \quad}l}
   1&    t=0 \\
 0 &t \geq 1 \\
     \end{array}\right.$}&A1&{\scriptsize {\bf $0.054$} sec.}\\  \cline{2-5}
   \raisebox{1.5 ex}[0pt]   {\bf U(2,3)} & $\vec{v}$=[[6, 1], [-1, -1, -5]]&{\scriptsize $m_{\mu+t\vec {v}} ^\lambda= \left\{
\begin{array}{r @{\quad {\text {if}} \quad}l}
   1&  t\leq 10 \\
 0 &t \geq 11 \\
     \end{array}\right.$}&A2&{\scriptsize {\bf $0.736$} sec.}\\   \cline{2-5}
     & $\vec{v}$=[[1, 0], [0, 0, -1]]&{\scriptsize $m_{\mu+t\vec {v}} ^\lambda=
\begin{array}{r @{\quad {\text {if}} \quad}l}
   1& t\geq 0 \\
     \end{array}$}&A3&{\scriptsize {\bf $0.26$} sec.}\\   \hline
  &$\lambda$=[[59, 39], [51, 7, -156]]&&&  \\  \cline{2-2}
{\bf U(2,3)}& $\mu_{lowest}$=[[121/2, 79/2], [51, 6, -157]]&&&\\  \cline{2-5}
& $\vec{v}$=[[1,0], [0,0,-1]]& 
$m_{\mu+t\vec {v}} ^\lambda=t+1    $&A4&{\scriptsize  {\bf $ 0.71$} sec.}\\  \hline
  &$\lambda$=[[341/2, 49/2], [-3/2, -5/2, -11/2, -371/2]]&&&  \\  \cline{2-2}
{\bf U(2,4)}& $\mu_{lowest}$=[[345/2, 53/2], [-5/2, -7/2, -13/2, -373/2]]&&&\\  \cline{2-5}
& $\vec{v}$=[[6, 1], [-1, -1, -1, -4]]& 
$m_{\mu+t\vec {v}} ^\lambda= \left\{
\begin{array}{r @{\quad {\text {if}} \quad}l}
   1&  t\leq 2 \\
 0 & t \geq 3 \\
     \end{array}\right.$&A5&{\scriptsize  {\bf $46.754$} sec.}\\  \hline
&$\lambda$=[[343/2, 31/2, 21/2], [-13/2,-19/2,-363/2]]&&&  \\  \cline{2-2}
& $\mu_{lowest}$=[[173, 17, 12], [-8, -11, -183]]&&&\\  \cline{2-5}
{\bf U(3,3)}& $\vec{v}$=[[6, 1, 0], [-1, -1, -5]]& $m_{\mu+t\vec {v}} ^\lambda= \left\{
\begin{array}{r @{\quad  {\text {if}}\quad}l}
   1& \t=0\\
3 & t =1 \\
6 & 2\leq t\leq 171 \\
3 & t=172 \\
1& t=173\\
0&t\geq 174\\
 \end{array}\right.$&A6&{\scriptsize {\bf $52.020$} sec.}\\ \cline{2-5}
& $\vec{v}$=[[1, 1, 0], [0,-1, -1]]& $m_{\mu+t\vec {v}} ^\lambda= \left\{
\begin{array}{r @{\quad  {\text {if}}\quad}l}
   t+1 &0\leq t \leq 155 \\
156&t\geq 156\\

     \end{array}\right.$&A7&{\scriptsize {\bf $7.1730$} sec.}\\  \hline

\end{tabular}
}

where

\medskip

\noindent{\scriptsize
\begin{tabular}{| l |}\hline
$A1:=[[-(1/2)*t^2+(1/2)*t+1, [0, 0]], [1+(1/2)*t^2-(3/2)*t, [1, 1]], [0, [2, inf]]]$\\ \hline
$A2:=[[-(15/2)*t^2+(7/2)*t+1, [0, 0]], [1, [1, 10]], [66+(1/2)*t^2-(23/2)*t, [11, 11]], [0, [12, inf]]]$\\ \hline
$A3:=[[-(1/2)*t^2+(1/2)*t+1, [0, 0]], [1, [1, inf]]]$\\ \hline
$A4:=[[t+1, [0, inf]]$\\
$A5:=[1+(10/3)*t-(7/2)*t^2+(25/6)*t^3, [0, 0]], [1+(1/3)*t-(1/2)*t^2+(1/6)*t^3, [1, 1]], [1, [1, 2]],$\\
\quad \quad \quad $[6-(7/2)*t+(1/2)*t^2, [2, 3]], [0, [3, inf]]$\\ \hline
$A6:=[1+(13/4)* t-(51/8 )*t^2+(61/4)* t^3-(89/8 )*t^4,[0,0]],[1+(3/2)* t+(1/2)* t^2,[1,1]],$\\
\quad \quad \quad $[-3+(27/4)* t-(1/8 )*t^2-(3/4)* t^3+(1/8 )*t^4,[2,3]],[6,[4,170]],$\\
\quad \quad \quad  $[32664996-(9380059/12)* t+(168227/24)* t^2-(335/12)* t^3+(1/24)* t^4,[171,172]],$\\
\quad \quad \quad $[15225-(349/2)*t+(1/2)*t^2, [172, 173]]$\\
\quad \quad \quad$[78155000-(10718575/6)*t+(183749/12)*t^2-(175/3)*t^3+(1/12)*t^4, [174, 175]], [0, [176, inf]]$\\ \hline
$A7:=[1+(5/4)*t+(1/24)*t^2-(1/4)*t^3-(1/24)*t^4, [0, 0]], [t+1, [1, 154]],$\\
\quad \quad \quad $[23726781-(2457885/4)*t+(143207/24)*t^2-(103/4)*t^3+(1/24)*t^4, [155, 156]], [156, [156, inf]]$\\ \hline
\end{tabular}
}
\medskip

Remark that in some of the examples above,  it can happen that although the polynomials $P_i$ and $P_{i+1}$
(displayed in  the last table giving the $A_i$)
are different, the polynomial $P_{i+1}$ may coincide with $P_i$ on $I_i$ (recall that  several polynomials can have the same values on $I_i\cap \N$).
Thus in this case, we join the two intervals $I_i$ and $I_{i+1}$ and give only the polynomial $P_{i+1}$.
This streamlining of the function $m_{\mu+t \vec v}^{\lambda}$  is given in the third column of the table describing the asymptotic behavior of   $m_{\mu+t \vec v}^{\lambda}$.

We now give an example on $U(3,4)$.

\begin{example}\label{2}\end{example}
{\scriptsize
\begin{verbatim}
discrete:=[[473, 39, 1], [3, 51, 5, -572]];
direction:=,[[1, 0, 0], [0, 0, 0, -1]];

>function_discrete_mul_direction_lowest(discrete,direction,3,4);

[1+51/20 t-1/120 t^5-1/360 t^6+851/360 t^2+23/24 t^3+5/36 t^4,[0,0]],
[1+31/12 t+19/8 t^2+11/12 t^3+1/8 t^4,[1,37]],
[-3262622+2687514/5 t+73/240 t^5-1/720 t^6-13275857/360 t^2+64795/48 t^3-3977/144 t^4,[38,39]],
[-265030+27790 t-1090 t^2+20 t^3,[39,44]],
[-9631849+79305707/60 t+29/80 t^5-1/720 t^6-3399664/45 t^2+110609/48 t^3-5675/144 t^4,[45,45]],
[27182687-212385511/60 t-31/40 t^5+1/360 t^6+69073219/360 t^2-132929/24 t^3+1619/18 t^4,[46,46]],
[-784945+886169/12 t-20959/8 t^2+511/12 t^3-1/8 t^4,[47,83]],
[469370132-337238937/10 t-167/240 t^5+1/720 t^6+363465857/360 t^2-774005/48 t^3+20897/144 t^4,[84,85]],
[5790400-235000 t+2820 t^2,[86,inf]]
\end{verbatim}} }

Thus the  multiplicity $m_{\mu+t \vec{v}}^{\lambda}$ can be completely described by the following  piecewise polynomial function:

{\scriptsize
$$m_{\mu+t\vec {v}} ^\lambda= \left\{
\begin{array}{l @{\quad  {\text {if}}\quad}l}
1+(31/12)*t+(1/8)*t^4+(11/12)*t^3+(19/8)*t^2& 0\leq t\leq 39\\
-265030+27790*t+20*t^3-1090*t^2&  40\leq t \leq 46 \\
-784945+(886169/12)*t-(1/8)*t^4+(511/12)*t^3-(20959/8)*t^2& 47\leq t  \leq 85 \\
5790400-235000*t+2820*t^2&86 \leq t. \\
 \end{array}\right.$$}

The time to compute the example is $TT := 19.487$ and the formula says for instance
that,   for $\lambda=$discrete and $\mu=[[475, 40, 0], [103/2, 9/2, 5/2, -1147/2]]$ the lowest $K$-type, then
$$m_{\mu+20000000\vec v} ^\lambda=1127995300005790400.$$

%

\section{ The program: "Discrete series and K multiplicities for  type $A_r$"}\label{programs}

We give a brief sketch of the main steps for the algorithms involved in  Blattner's formula.
\subsection{$MNPS$ non compact}
We outline the algorithm that computes directly  $\overrightarrow{M}$ for $M\in \CP(v,\Delta^+(A,B)).$ We are taking advantage of the fact that we know the $\overrightarrow{M}'s$ in the case of $|A|=1,\  \text r$, as we saw in Ex. \ref{ex}.
 In the following scheme $p,q$ are integers, $A\subset [1,2,\ldots,p+q]$ is a set  of cardinality $p$, $B$ is the complement subset  defining $U(A,B)$,  $\theta_I$ is the highest noncompact root for $I$. If $L\subset [1,2,\ldots,p+q]$ we denote by $L'$ the complement set.
\begin{figure}[ht]
\begin{center}
\scriptsize{
\begin{tabular}{|ll|}\hline
Input     $[v,A,I]$, $v$  a vector  and  $A\subset I= [1,2,\ldots,p+q]$, $|A|=p$&\\
proceed by induction on the  cardinality of $A$.&\\
 \quad if $|A|=1$ \= or $ |A|=r$ write the unique  $\overrightarrow{M}$, $M\in MNPS$ determined by the situation&\\
  \quad\quad  if $ v\in C(\overrightarrow{M})$ then the output is   $\overrightarrow{M}$,  &\\
\quad construct the hyperplane $H_L$.&\\
  \quad check \=    if  $v$ and $\theta_I$ are on the same side then $H_L$& \\
   \quad\quad if not skip the hyperplane&\\
   \quad  define the projection $v'=\proj_{H_{L}}(v)$ of $v$ on $H_L$ along $\theta_I$&\\
  \quad   compute  $[v_1,A_1,I_1], [v_2,A_2,I_2]$&  \\
 \quad where $A_1=L\cap A$, $A_2= L'\cap A$, $I_1=L$, $I_2=L'$ and &\\
   \quad $ v_1,v_2$ \=  are the components of $v'$  on $L$ and $L'$ respectively.&\\
 \quad\quad if   $v_1$\= (resp. $v_2$)  is not in the positive cone for $\Delta(I_1)$  then skip the hyperplane&\\
    \quad\quad\quad if $ |A_1|=1$,  apply the induction  and compute  $\overrightarrow{M_1}$, $M_1 \in MNPS(v_1,A_1,I_1)$, &  \\
\quad\quad\quad add to  $M_1$ the root $\theta_{I_1}$ (do the same if $|A_2|=1$ )&\\
  \quad\quad\quad else apply the induction and compute  $\overrightarrow{M_i}$, $M_i \in \CP(v_i,A_i,I_i), i=1,2$&\\
  \quad\quad\quad do the cartesian product  $\overrightarrow{M_1} \times \overrightarrow{M_2}$ and add to each set  the root $\theta_I$ &\\
  \quad \quad collect all $\overrightarrow{M}'s$,  for the wall $L$&\\
     \quad end of loop running across $L$'s&\\
   \quad  end induction&\\
   return the set of all $\overrightarrow{M}'s$, $ M\in MPNS$  for all hyperplanes&\\
  \hline
\end{tabular}}
\end{center}
\caption{$\CP(v,\Delta^+(A,B))$}
\label{algo.nonc}
\end{figure}

\subsection{Numeric}\label{algo.num}
The scheme is described in Fig.\ref{Blat.num}.
\begin{figure}[h!]

{\scriptsize{\bf Subroutines}:
 \begin{itemize}
\item Procedure to find $\CA^+$-admissible hyperplanes.
\item Procedure to deform a vector: $DefVec_nc(v, \CA^+)$
\item Procedure to compute $\overrightarrow{M}$, $M\in \CP(v,\CA^+)$ as  in Fig.\ref{algo.nonc}.
\item Procedure to compute Kostant function $K(h)=K(0,h)=\frac{e^{h}}{\prod_{\alpha\in \Delta_n^+}(1-e^{-\alpha})}$ or
more generally  $K(g,h).$
\item Compute the valid permutation $Valid(u,v)\subset  {\mathcal W}_c$
\end{itemize}}
\scriptsize{
\begin{tabular}{|ll|}\hline
Input:  $\lambda$, $\mu$&\\
   Compute $\CA^+=\Delta_n^+(\lambda)$:&\\
   Compute $Valid(\lambda,\mu)$\\
  \quad   for  each $w\in Valid(\lambda,\mu)$, \\
  \quad compute $\mu_{reg}=DefVec_{nc}(v, \CA^+)$\\
  \quad compute   $All_w:=\CP(w\mu_{reg} -\lambda,\Delta^+_n)$&\\
\quad   compute $cont_w=\sum_{\overrightarrow{M}\subset {\it All_w}} \Ires_{\overrightarrow{M}} K(w\mu-\lambda-\rho_n)$&\\
   \quad end of loop running across $w$'s&\\
collect all the terms and return& \\
$m^\lambda_\mu= \sum_{w\in Valid(\lambda,\mu)}\epsilon(w) cont_w$&\\
\hline
\end{tabular}}
\caption{Blattner's algorithm (numeric case)} \label{Blat.num}
\end{figure}

\subsection{Asymptotic directions}\label{algo.asym}
We fix the parameter  $\lambda_0$ and  $\mu_0$, regular in the chambers $\a,\a_c$ and a weight $\vec{v}.$ We want to compute $m^{\lambda_0}_{\mu+t\vec{v}}.$ In the application $\mu_0$ will be the lowest $K$-type.
The scheme is described in Fig.\ref{algo.pol}.
\begin{figure}[h!]

{\scriptsize{\bf Subroutines}:
 \begin{itemize}
\item Procedure to find $\CA^+$-admissible hyperplanes.
\item Procedure to deform a vector: $DefVec_{nc}(v, \CA^+)$
\item Procedure to compute $\overrightarrow{M}$, $M\in |CP(v,\CA^+)$ as  in Fig.\ref{algo.nonc}.
\item Procedure to compute Kostant function $K(h)=K(0,h)=\frac{e^{h}}{\prod_{\alpha\in \Delta_n^+}(1-e^{-\alpha})}$ or
more generally  $K(g,h).$
\end{itemize}}
\scriptsize{
\begin{tabular}{|ll|}\hline
Input  $\lambda_0$ and  $\vec{v}$
\quad for  each $H$ noncompact wall &\\
\quad\quad if $(H,\vec{v})= 0$ then skip $H$&\\
  \quad\quad else if $(H,\lambda_0-w\mu_0)(H,\vec{v})<0$ then skip $H$ else&\\
 \quad\quad collect $ t_H=(H,\lambda_0-w\mu_0)/(H,\vec{v})$&\\
\quad end of loop running across $H$'s&\\
\quad order list $t_H's$ as $[t_0,t_1,\ldots,t_s]$ where $t_0=0,t_s=\infty$&\\
\quad choose an interior point $\bar{ t_i}$ in  each interval $[t_i,t_{i+1}]$&\\
 \quad Compute polynomial on each $[t_i,t_{i+1}]$ following the scheme Fig.\ref{Blat.num} and the ordered basis determined by $\bar{ t_i}$&\\
 output: the sequence of values $m^{\lambda_0}_{\mu_0+t\vec{v}}$ valid on $[t_i,t_{i+1}], \ t\in \N$&\\
 \hline
\end{tabular}}
\caption{Blattner's algorithm: asymptotic case}
\label{algo.pol}.
\end{figure}

\end{document}